\documentclass[11pt,reqno,oneside]{amsart}
\usepackage[a4paper, total={6in, 8in}]{geometry}
\usepackage{mathtools}  
\usepackage{thmtools}
\usepackage{thm-restate}
\mathtoolsset{showonlyrefs} 
\usepackage[usenames]{color}
\usepackage{amsmath,amsfonts,amssymb,pdfsync,verbatim,graphicx,epstopdf,enumerate,fancybox}
\usepackage[colorlinks=true]{hyperref}
\usepackage{mathrsfs}
\usepackage{cancel}

\hypersetup{urlcolor=red, citecolor=red, linkcolor=blue}
\usepackage{url}
\usepackage[notref,notcite]{}
\usepackage[framemethod=tikz]{mdframed}

\theoremstyle{plain}
\newtheorem{theorem}{Theorem}[section]

\theoremstyle{plain}

\theoremstyle{plain}
\newtheorem{lemma}[theorem]{Lemma}

\theoremstyle{plain}
\newtheorem{proposition}[theorem]{Proposition}

\theoremstyle{plain}
\newtheorem{corollary}[theorem]{Corollary}

\theoremstyle{definition}
\newtheorem{definition}[theorem]{Definition}

\theoremstyle{definition}
\newtheorem{remark}[theorem]{Remark}

\theoremstyle{plain}

\theoremstyle{plain}

\theoremstyle{plain}

\numberwithin{equation}{section}

\def\tteal{\textcolor{black}}

\usepackage{amsmath,pdfsync,verbatim,graphicx,epstopdf,enumerate,mathtools,bbm}

\usepackage[centerlast,small,sc]{caption}
\usepackage[usenames]{color}
\usepackage{amsmath,pdfsync,verbatim,graphicx,epstopdf,enumerate}

\newcommand{\Ec}{\mathcal{E}}

\newcommand{\Sc}{\mathcal{S}}

\newcommand{\Sb}{\mathbb{S}}
\newcommand{\Sn}{\mathbb{S}^{n-1}}

\newcommand{\Rn}{\mathbb{R}^n}
\newcommand{\Rb}{\mathbb{R}}
\newcommand{\D}{\mathrm{d}}

\newcommand{\Beq}{\begin{equation}}
\newcommand{\Eeq}{\end{equation}}

\newcommand{\dsx}{S_{\xi}}

\newcommand{\bfi}{\mathbf{i}}

\DeclareRobustCommand{\SkipTocEntry}[5]{}

\title[Unique continuation for the momentum ray transform]{Unique continuation for the momentum ray transform}

\author[Ilmavirta, Kow and Sahoo]{Joonas Ilmavirta, Pu-Zhao Kow and Suman Kumar Sahoo}

\address{Department of Mathematics and Statistics, University of Jyv\"askyl\"a, Finland.
\newline
E-mail:{\tt\ joonas.ilmavirta@jyu.fi}
\newline
ORCID: {\tt\ \tteal{0000-0002-2399-0911}}
}

\address{\tteal{Department of Mathematical Sciences, National Chengchi University, Taiwan.}
\newline
E-mail:{\tt\ \tteal{pzkow@g.nccu.edu.tw}}
\newline
ORCID: {\tt\ \tteal{0000-0002-2990-3591}}}

\address{\tteal{Department of Mathematics, Indian Institute of Technology, Bombay.}
\newline
E-mail:{\tt \ \tteal{suman@math.iitb.ac.in}}
\newline
ORCID: {\tt\ \tteal{0000-0002-6459-1597}}}

\begin{document}

\begin{abstract}
The present article focuses on a unique continuation result for certain weighted ray transforms, utilizing the unique continuation property (UCP) of the fractional Laplace operator. Specifically, we demonstrate a conservative property for momentum ray transforms acting on tensors, as well as the antilocality property for both weighted ray and cone transforms acting on functions.
\end{abstract}

\subjclass[2020]{Primary 46F12, 45Q05, 35R11}
\keywords{ray transform, tensor tomography, antilocality property, unique continuation property (UCP), fractional Laplacian, fractional elliptic operator, Saint Venant operator}

\maketitle

\tableofcontents

\begin{sloppypar}

\section{Introduction} 


To what extent is a tensor field determined by its line integrals?
Depending on the weighting and other choices in the line integrals, there may or may not be a gauge freedom or other non-uniqueness.
We will focus on the so-called momentum ray transform, and study
\tteal{an inverse problem with partial data}
with the help of the unique continuation principle or property (UCP).

The space of $m$-tensor fields in $\mathbb{R}^{n}$ is denoted by $T^{m} \equiv T^{m}(\mathbb{R}^{n})$, while its subspace of covariant symmetric $m$-tensor fields on $\mathbb{R}^{n}$ is denoted by $S^{m} \equiv S^{m}(\mathbb{R}^{n})$.  In Cartesian coordinates, an element $f$ can be written as
$$ f(x) = f_{i_1\dots i_m}(x)\, \D x^{i_1} \cdots \D x^{i_m},$$
where  $f_{i_1 \dots i_m}(x)$ is symmetric in all indices $i_1, \dots, i_m \in \{1,\dots,n\}$.
For repeated indices, Einstein summation convention will be assumed throughout this article. 

Let $ \Sc(S^m)= \Sc(\Rn; S^m(\Rn))$ be the Schwartz  class of symmetric $m$ tensor fields in $\Rn$.  We denote $\Sc(S^0)= \Sc (\Rn)$.
\tteal{The integral transforms that we focus on in our study can be defined on these spaces as follows:}
\begin{enumerate}
    \item 
    The ray transform $If$ of $f\in \Sc(S^m)$ is defined by
\begin{align*}
    If(x,\xi) := \int\limits_{-\infty}^{\infty}  f_{i_1\cdots i_m}(x+t\xi)  \, \xi_{i_1}\cdots \xi_{i_m} \D t.
\end{align*}
    \item
    The $k^{\rm th}$ momentum ray transform (MRT) $I^kf$ of $f\in \Sc(S^m)$ is defined by
\begin{align*}
    I^kf(x,\xi):= \int\limits_{-\infty}^{\infty} t^k\, f_{i_1\cdots i_m}(x+t\xi)  \, \xi_{i_1}\cdots \xi_{i_m} \D t \quad \mbox{for all integers $k\ge 0$}.
\end{align*}
    \item
    For $s\in(0,\frac{n}{2})$, the fractional momentum ray transform $I^{2s-1}f$ of $f\in \Sc(\Rn)$ is defined by
    \begin{equation}
\big( \mathcal{X}_{s}f \big) (x,\xi) := \big( I^{2s-1} f \big) (x,\xi) \equiv \int\limits_{0}^{\infty} t^{2s-1} f(x + t \xi)\, \D t. 
\end{equation}
\end{enumerate}

\tteal{It is easy to see that the fractional momentum ray transform $I^{2s-1}$ reduces to the classical momentum ray transform when $2s-1=k$ is a non-negative integer.} In the case $k=0$ the momentum ray transform reduces to the classical ray transform: $I^0=I$.
This transform has been  studied extensively due to its wide range of potential applications in various scientific fields; see \cite{PSU_book}. MRTs
were first introduced by Sharafutdinov \cite{Sharafutdinov_Book} and studied further in the works \cite{Krishnan2018,Krishnan2019a,SumanSIAM}. For functions, $I^kf$ appears to give the inversion formula of cone transform and conical Radon transform. \tteal{The conical Radon transform} has promising application in Compton cameras; see \cite{Kuchment_IPI,Parra2000,Smith2005,Tomitani2002}. For tensors, MRTs appear to solve Calder\'on type inverse problems for polyharmonic operators; see \cite{BKS,SS_linearize_polyharmonic}. Another motivation to study  MRTs are  its connection with the exponential ray transform defined as $(I_{\exp}^{\alpha}f)(x,\xi) := \int_{\Rb} e^{\alpha t } f_{i_1\cdots i_m}(x+t\xi)  \, \xi_{i_1}\cdots \xi_{i_m} \D t$, for some real number $\alpha$.
Formally (and truly whenever the series is convergent in a suitable sense) the exponential ray transform is related to the momentum ray transforms by
\[ I_{\exp}^{\alpha} f = \sum_{k=0}^{\infty} \frac{\alpha^k}{k!} I^kf. 
\] 
We now state the main results (not all) of our article. We refer the reader Section~\ref{sec:preliminary} for the definition of \tteal{the generalized Saint-Venant operator $R^k$ (see \eqref{GSV-equivalent}) and the normal operator $N^{p}$ (see \eqref{eq:normal_operator}) .} 
\tteal{In the following theorems, one should think of measurements being made only in the set $U$. A more detailed discussion of partial data results and their connection to UCP is given after the theorems.}

\begin{restatable}{theoremx}{goldbach}
\label{mrt_ucp} 
Let $m \in \mathbb{N}$ and an integer $0 \le k \le m$. Suppose that 
\begin{equation}
\begin{cases}
f \in L^{1}(\mathbb{R}^{n};S^{m}) \cap L_{\rm loc}^{2}(\mathbb{R}^{n};S^{m}) & \text{when $0 \le k \le n-1$,} \\
f \in L^{2}(\mathbb{R}^{n};S^{m}) \text{ with compact support} &\text{when $n \le k \le m$,}
\end{cases} \label{eq:assumption-f-normal-well-defined}
\end{equation}
\tteal{and assume that} there exists an nonempty open set $U$ in $\mathbb{R}^{n}$ such that  
\[
R^{k} f|_U = 0 \quad \text{for some $0 \leq k \leq m$. }
\]
\tteal{If} there exists $x_{0} \in U$ such that for each $0 \leq p \leq k$ the following assumption holds: 
\begin{equation}
\text{$N^{p}f$ vanishes at $x_{0}$ of infinite derivative order,} \label{eq:vanishes-infinite-derivative-order}
\end{equation}
then $R^{k}f \equiv 0$ in $\mathbb{R}^{n}$. If we additionally assuming that $f$ has compact support, then $f$ is a generalized potential field, that is, $f= \D^{k+1}v$ for some $v\in \Ec'(\Rn; S^{m-k-1})$. We refer \eqref{GSV-equivalent} and \eqref{eq:normal_operator} for the definitions of $R^k$ and $N^k$ respectively.
\end{restatable}



\begin{restatable}{theoremx}{measurable}\label{thm:MUCP-ray-transform-function}
    Suppose $U \subseteq \Rb^n$ be any non empty open set and $ n\ge 2$. Suppose that $N^{0}f \in H^{\frac{1}{2}}(\mathbb{R}^{n})$. If $f|_{U}=0$ and there exists a positive measure set $E$ in $U$ such that $\left( N^{0}f \right)|_{E}=0$, then $f \equiv 0$ in $\mathbb{R}^{n}$.
\end{restatable}
 

\tteal{We refer the reader Section~\ref{subsec:ray-transform-elliptic-operator} for the definition of the average operator $\mathcal{A}_{s}$, see for instance equation \eqref{def_A_s}.}

\begin{restatable}{theoremx}{fractional}\label{thm:UCP-open-Xray} 
Let $n \ge 2$ be an integer, $0 < s \le \frac{n}{4}$ with $s \neq \mathbb{Z}$ and let $f \in L^{\frac{2n}{n+4s}}(\mathbb{R}^{n})$. If there exists a non-empty open set $U$ in $\mathbb{R}^{n}$ such that \tteal{$f = \mathcal{A}_{s}f = 0$ in $U$, then $f=0$ in $\Rn$}, where $(\mathcal{A}_s f)(x)= c(n,-s) \int_{\Sn} (\mathcal{X}_s f)(x,\xi)\, \D S(\xi)$ with $c(n,s) := \frac{2^{2s} \Gamma(\frac{n+2s}{2})}{\pi^{n/2}|\Gamma(-s)|}$.
\end{restatable}




The following table contains the list of some existing results in this direction:\smallskip

\begin{tabular}{ |p{3cm}||p{3cm}|p{4cm}|p{3cm}|  }
 \hline
 \multicolumn{4}{|c|}{Uniqueness results} \\
 \hline
Transforms & Full data & UCP & Support theorem\\
 \hline
 Ray transform   &  \cite{Helgason_Book,Sharafutdinov_Book}   & \cite{Keijo_partial_function,Keijo_partial_vector_field,DVS_UCP_2022}&\cite{Krishnan2007}\\
 $d$ plane transform & \cite{Helgason_Book}  & \cite{CMR_ucp}& \cite{Helgason_Book}\\
 MRT& \cite{Krishnan2018}  & \cite{DVS_UCP_2022}, current article&\cite{Abhishek-Mishra}\\
 Cone transform& \cite{Kuchment_IPI}& current article & unknown\\
\hline
\end{tabular}
\smallskip

The ray transform with full data is a classical concept, and we have only given text book references rather than original articles for it.

The unique continuation result for ray transform can be considered as a partial data uniqueness result, in which one considers the following question:
If a function or a symmetric tensor field $f$ satisfies $If=0$ and $f=0$ in some open set $U$, 
then can one conclude that $f=0$ (for functions) or $f$ is a potential field (for symmetric tensor fields) in the whole space?

This question is related to the \emph{interior tomography problem} or the \emph{region of interest (ROI) tomography} problem. 
The goal of ROI tomography is to reconstruct the function within an open set (the region of interest) from the line integrals over all lines through that set.
It is well known that from this information one cannot recover the whole function in the region of interest; see \cite[Example 2.1]{Klann_Quinto_IP} 
However, it is possible to recover the singularities of the function in the ROI because of the fact that ray transform $If$ is an elliptic Fourier integral operator (FIO) and its normal operator $I^*If$ is an elliptic pseudo-differential operator.
Therefore by pseudo-local property, we have $\mbox{WF}(f)=\mbox{WF}(I^*If)$, see \cite{quinto_book_chapter}.
This fact serves as a motivation for Lambda tomography \cite{Smith_local_tomography}, which is a local reconstruction method that uses ROI  data to reconstruct singularities.


However, the problem we pose is different:
We only use the rays that meet the set $U$, but our goal is to reconstruct the scalar or tensor field outside $U$, not inside it.
We assume that something about the field is known in $U$.
The simplest choice is $f|_U=0$, but this can be relaxed to the gauge invariant assumption $Rf|_U=0$, where $R$ is the Saint Venant operator.
The assumption that the integrals over all lines vanish can be relaxed to $(I^k)^*I^kf$ vanishing in a subset of positive measure of $U$ or vanishing to infinite order at a point of $U$.
\tteal{We need to ensure $(I^k)^*I^kf$ to be smooth so that the assumption ``vanishing to infinite order'' make sense: in fact this} follows from ellipticity of the normal operator $(I^k)^*I^k$ as a pseudo-differential operator.

The following are equivalent for  $f\in \Sc(S^m)$ and $0\leq k\leq m$:
\begin{enumerate}
\item $I^kf=0$ for all lines
\item $R^kf=0$ (the $k^{\rm th}$ power of the Saint-Venant operator)
\item $f$ is the $(k+1)^{\rm st}$ symmetrized covariant derivative of some $v\in\Sc(S^{m-k-1})$ (with the convention that tensor fields of order $-1$ are identically zero)
\end{enumerate}
To prove a similar conclusion for our partial data setting, we aim to pass from partial data on $I^kf$ to the global conclusion $R^kf=0$, from which the conclusion follows.
The advantage of  working with $R^kf$ is that it is a differential operator which is local in nature --- unlike the integral operator $I^k$.

By connecting the normal operator of ray transform with fractional Laplacian, we see that the UCP for ray transforms is a consequence of UCP (more precisely, the antilocality property) of fractional Laplacian. UCP (Unique Continuation Property) is a useful tool for studying inverse problems, where one aims to recover an unknown function from measured data. UCP ensures that if a solution to a PDE vanishes in  a suitable sense, then it must be identically zero. UCP has been extensively studied  for local operators (e.g., Laplace and wave operators) \cite{IsakovPartialData07} and more recently for nonlocal operators (e.g., fractional Laplace equation,  and fractional wave equation of peridynamic type) \cite{GSU-fractional,Tuhin_Ibero}. Interestingly, the proof of UCP for some nonlocal elliptic operators requires the use of UCP for local elliptic equations, see \cite{GR19unique}.

The rest of the article is organized as follows. In Section~{\rm \ref{sec:preliminary}}, we provide a recap of some preliminary results, including definitions and notation that we will use in the upcoming sections. In Section~{\rm \ref{sec:saint Venant}},  we study more properties of \tteal{the} generalized Saint Venant operator, and  a new decomposition theorem of symmetric tensor fields is presented in Section \ref{sec:Decomposition-solenoidal-potential}. Finally, we recall our main results and prove them in Section \ref{sec:main results}. In Appendix \ref{appen:Negative-power-elliptic}, we provide the domain of a negative fractional power of a general elliptic operator (see Definition \ref{def:domain_of_anisotroipic_fractional_operator}) analogously to the fractional Laplace operator.  




\addtocontents{toc}{\SkipTocEntry}
\section*{Acknowledgements}

J. I. was supported by the Academy of Finland (grants 332890 and 351665). \tteal{P.-Z. K was supported by the NCCU Office of research and development and the National Science and Technology Council of Taiwan (grant NSTC~112-2115-M-004-004-MY3).} P.-Z. K and S. K. S were partly supported by the Academy of Finland (Centre of Excellence in Inverse Modelling and Imaging, grant 312121) and by the European Research Council under Horizon 2020 (ERC CoG 770924).
\tteal{We thank the anonymous referees for their comments.}

\section{Preliminaries}\label{sec:preliminary}

In order to make the paper self-contained, in this section, we introduce some operators -- especially the momentum ray transforms and Saint-Venant operators -- as well as some of their known properties. We are not going to exhaust all the details here, see e.g.\ \cite{DVS_UCP_2022,Krishnan2018,Sharafutdinov_Book} and references therein for more details. Readers who are already familiar with the subject may proceed to the next section.


\addtocontents{toc}{\SkipTocEntry}
\subsection{Momentum ray transform on Schwartz space}

For each vectors $\xi^{(1)},\cdots,\xi^{(m)} \in \mathbb{R}^{n}$, their \emph{tensor product} (or \emph{juxtaposition}) $\xi^{(1)} \otimes \cdots \otimes \xi^{(m)}$ is defined by 
\[
( \xi^{(1)} \otimes \cdots \otimes \xi^{(m)} )_{i_{1}\cdots i_{m}} := \xi_{i_{1}}^{(1)} \cdots \xi_{i_{m}}^{(m)}.
\]
If $\xi^{(1)} = \cdots = \xi^{(m)} = \xi$, we simply denote 
\[
\xi^{\otimes m} := \xi^{(1)} \otimes \cdots \otimes \xi^{(m)} \in S^{m}\quad \text{and} \quad \xi_{i_{1}\cdots i_{m}}^{\otimes m} := (\xi^{\otimes m})_{i_{1}\cdots i_{m}}. 
\]
We also denote $\langle \cdot , \cdot \rangle : S^{m} \times S^{m} \rightarrow \mathbb{R}$ by 
\[
\langle f,g \rangle := f_{i_{1}\cdots i_{m}}g_{i_{1}\cdots i_{m}}.
\]
Let ${\Sc}(\mathbb{R}^{n};S^m)$ be the Schwartz class of symmetric $m$ tensor fields in $\Rn$. Given any non-negative integer $k$, we define the mapping $J^k :{\Sc}(\mathbb{R}^{n};S^m) \rightarrow C^{\infty}(\mathbb{R}^{n} \times (\mathbb{R}^{n} \setminus \{0\}))$ by
\begin{equation}\label{eq:definition of extended momentum ray transform}
(J^k f)(x,\xi) := \int\limits_{-\infty}^\infty t^k\,\langle f(x+t\xi),\xi^{\otimes m}\rangle \, \D t \equiv \int\limits_{-\infty}^\infty t^k\, f_{i_1\dots i_m}(x+t\xi)\,\xi_{i_1} \cdots \xi_{i_m} \, \D t
\end{equation}
for all $(x,\xi) \in \mathbb{R}^{n} \times (\mathbb{R}^{n} \setminus \{0\})$. Denote the tangent bundle of the unit sphere by 
\[
T \mathbb{S}^{n-1} := \begin{Bmatrix}\begin{array}{l|l} (x,\xi) \in \mathbb{R}^{n} \times \mathbb{S}^{n-1} & \langle x,\xi \rangle = 0 \end{array}\end{Bmatrix}.
\]
Since each point $(x,\xi)\in T{\Sb}^{n-1}$ determines a unique line $x+t\xi$ with $t\in \Rb$, then it  make sense to denote $I^{k} \equiv J^{k}|_{T \mathbb{S}^{n-1}}$ be the restriction of $J^{k}$ on $T \mathbb{S}^{n-1}$.
Let $\mathcal{S}(T \mathbb{S}^{n-1})$ be the space of smooth functions $\varphi(x,\xi)$ on $T \mathbb{S}^{n-1}$ such that all their derivatives decrease rapidly in the first argument \cite[Section~2.1]{Sharafutdinov_Book}. \tteal{It is easily seen that}  
\begin{equation}
I^{k} \equiv J^{k}|_{T \mathbb{S}^{n-1}} : \mathcal{S}(\mathbb{R}^{n};S^{m}) \rightarrow \mathcal{S}(T \mathbb{S}^{n-1}) \quad \mbox{is bounded}. \label{eq:definition of momentum ray transform}
\end{equation}

\begin{definition} \label{def:MRT-schwartz}
For each $m \in \mathbb{Z}_{\ge 0}$ and integer $0 \le k \le m$, we call \eqref{eq:definition of momentum ray transform} the \emph{$k^{\rm th}$ momentum ray transform} of a symmetric $m$-tensor field. 
\end{definition}


\begin{remark}[An equivalent relation]
From \cite[(2.6)]{Krishnan2018} we have 
\begin{equation}\label{eq:relation between Ik and Jk}
(J^q f)(x,\xi) = |\xi|^{m-2q-1} \sum_{\ell=0}^q (-1)^{q-\ell} 
\binom{q}{\ell} |\xi|^{\ell}  \langle\xi,x\rangle^{q-\ell}  (I^\ell f)
\underbrace{\left( x - \frac{\langle x, \xi \rangle}{|\xi|^2} \xi , \frac{\xi}{|\xi|}\right)}_{\in \, T\Sb^{n-1}}
\end{equation}
\tteal{for all $(x,\xi) \in \mathbb{R}^{n} \times (\mathbb{R}^{n} \setminus \{0\})$}, which implies, for each integer $0 \le k \le m$, that 
\[
\begin{aligned}
&(I^{0}f,I^{1}f,\cdots,I^{k}f)=(0,0,\cdots,0) \:\tteal{\mbox{if and only if}} \:(J^{0}f,J^{1}f,\cdots,J^{k}f)=(0,0,\cdots,0). 
\end{aligned}
\]
\end{remark}

\addtocontents{toc}{\SkipTocEntry}
\subsection{Momentum ray transform on compactly supported tensor field distributions\label{subsec:assumption-regularity-MUCP}}

It is well-known that the (classical) ray transform $I^{0}$ is well-defined on compactly supported tensor field distributions \cite{Sharafutdinov_Book}. Similar extension for momentum ray transforms was considered in \cite{BKS} in order to study the polyharmonic operator inverse problem. Similar ideas also work for momentum ray transforms \cite[Section~2.3]{DVS_UCP_2022}, which will be presented here in order to make the paper self-contained. Let $\Ec'(\mathbb{R}^{n};S^{m})$ be the compactly supported symmetric $m$ tensor distribution fields in $\mathbb{R}^{n}$, and we define the bounded linear operator $I^k: \Ec'(\Rb^n; S^m) \to \Ec'(T\Sn)$ by
\begin{equation}
( I^kf , g )_{T \mathbb{S}^{n-1}} = \left( f , (I^k)^* g \right)_{\mathbb{R}^{n};S^{m}} \quad \text{for all $f \in \Ec'(\Rb^n; S^m)$ and $ g \in C^{\infty}(T\Sn)$,}
\end{equation}
where the distributional adjoint $(I^{k})^{*} : C^{\infty}(T\Sn) \rightarrow C^{\infty}(\mathbb{R}^{n};S^{m})$ is defined by 
\begin{equation}
\left( (I^{k})^{*}g \right)_{i_{1}\cdots i_{m}} (x) := \int\limits_{\mathbb{S}^{n-1}} \langle x,\xi\rangle^k \xi_{i_1} \dots \xi_{i_m} \, g(x-\langle x,\xi\rangle\xi,\xi) \, \D \dsx \quad \text{for all $x \in \mathbb{R}^{n}$.} \label{eq:distributional-adjoint-Ik-star}
\end{equation}
Similarly, we define $J^{k} : \Ec'(\mathbb{R}^{n};S^{m}) \rightarrow \mathcal{D}'(\mathbb{R}^{n} \times \mathbb{S}^{n-1})$ by 
\[
( J^{k}f,g )_{\mathbb{R}^{n} \times \mathbb{S}^{n-1}} = \left( f,(J^{k})^{*}g \right)_{\mathbb{R}^{n};S^{m}} \quad \text{for all $f\in \Ec'(\mathbb{R}^{n};S^{m})$ and $g \in C_{c}^{\infty}(\mathbb{R}^{n}\times \mathbb{S}^{n-1})$,}
\]
where the distributional adjoint $(J^{k})^{*}: C_{c}^{\infty}(\mathbb{R}^{n} \times \mathbb{S}^{n-1}) \rightarrow C^{\infty}(\mathbb{R}^{n};S^{m})$ is defined by 
\[
\left( (J^k)^{*} g \right)_{i_1 \cdots i_m}(x) := \int\limits_{\Sb^{n-1}}\int\limits_{\Rb} t^{k}g(x-t\xi, \xi)\xi_{i_1} \cdots \xi_{i_m}\, \D t\, \D \dsx \quad \text{for all $x \in \mathbb{R}^{n}$.}
\]

\addtocontents{toc}{\SkipTocEntry}
\subsection{Normal operator of momentum ray transforms}
Using \eqref{eq:definition of momentum ray transform}, for each integer $0 \le k \le m$, let us denote $N^k= (I^k)^*I^k : \mathcal{S}(\Rb^n; S^m) \to C^{\infty}(\Rb^n; S^m)$ the \emph{normal operator} of the $k^{\rm th}$ momentum ray transform in Definition~{\rm \ref{def:MRT-schwartz}}. For each $f \in \mathcal{S}(\mathbb{R}^{n};S^{m})$, we have
\begin{equation}
\begin{aligned}
&(N^{k}f)_{i_{1} \cdots i_{m}} (x) \\
&\quad = 2 \sum\limits_{\ell=0}^k \binom{k}{\ell} (-1)^{\ell} x_{p_{1} \cdots p_{2k-\ell}}^{\otimes (2k-\ell)} \left( f_{j_{1} \dots j_{m}} * \Xi_{p_{1} \cdots p_{2k-\ell} i_{1} \cdots i_{m} j_{1} \dots j_{m}} \right)(x),
\end{aligned} \label{eq:normal_operator}
\end{equation}
for all $x \in \mathbb{R}^{n}$, where  
\[
\Xi_{p_{1} \cdots p_{2k-\ell} i_{1} \cdots i_{m} j_{1} \dots j_{m}}(z) := \frac{ z_{p_{1} \cdots p_{2k-\ell} i_{1} \cdots i_{m} j_{1} \dots j_{m}}^{\otimes (2m+2k-\ell)}}{|z|^{2m+2k-2\ell +n-1}} \quad \text{for all $z \in \mathbb{R}^{n}$,}
\]
see \cite[(2.13)]{DVS_UCP_2022}, therefore the normal operators extend to the mapping
\begin{equation}
(N^{0},\cdots,N^{m}) : \Ec'(\mathbb{R}^{n};S^{m}) \rightarrow (\mathcal{S}'(\mathbb{R}^{n};S^{m}))^{m+1} \label{eq:extension-normal-compactly-distribution}
\end{equation}
as the convolution of a compactly supported distribution and a tempered distribution. Since 
\[
|\Xi_{p_{1} \cdots p_{2k-\ell} i_{1} \cdots i_{m} j_{1} \dots j_{m}}(z)| \le \frac{1}{|z|^{n-1-\ell}} \quad \text{for all $\ell \in \mathbb{Z}_{\ge 0}$,}
\]
then the convolution with $\Xi_{p_{1} \cdots p_{2k-\ell} i_{1} \cdots i_{m} j_{1} \dots j_{m}}(z)$ is \tteal{a mapping} from $L^{p}(\mathbb{R}^{n})$ to $L^{q}(\mathbb{R}^{n})$ provided $\frac{1}{q} = \frac{1}{p} - \frac{1+\ell}{n}$ when $0 \le \ell \le n-2$, see \eqref{eq:Riesz-domain-range}. When $\ell=n-1$, One can see that the convolution with $\Xi_{p_{1} \cdots p_{2k-\ell} i_{1} \cdots i_{m} j_{1} \dots j_{m}}(z)$ is a mapping from $L^{1}(\mathbb{R}^{n})$ to itself. Therefore the normal operators also extend to the mapping
\begin{equation}
N^{k} : L^{1}(\mathbb{R}^{n};S^{m}) \rightarrow \mathcal{D}'(\mathbb{R}^{n};S^{m}) \quad \text{for all integer $0 \le k \le n-1$.} \label{eq:extension-normal-L1}
\end{equation}
When $m = k = 0$, \eqref{eq:normal_operator} becomes 
\begin{equation}
\left( N^{0}|_{S^{0}}f \right) = 2 f * \frac{1}{|\cdot|^{n-1}} \quad \text{for all $f \in \Ec'(\mathbb{R}^{n}) \cup \mathcal{S}(\mathbb{R}^{n})$,} \label{eq:normal_operator_S0}
\end{equation}
and for each integer $\ell \ge 0$ we know that 
\begin{equation}
N^{0}|_{S^{0}} : W^{\ell,p}(\mathbb{R}^{n}) \rightarrow W^{\ell,q}(\mathbb{R}^{n}) \text{ is bounded provided } \frac{1}{q} = \frac{1}{p} - \frac{1}{n}, \label{eq:Riesz-domain-range2}
\end{equation}
see \eqref{eq:Riesz-domain-range1}. 

\addtocontents{toc}{\SkipTocEntry}
\subsection{Some differential operators}
We now recall  certain differential operators from \cite[Section~2.1]{Sharafutdinov_Book}. The \emph{symmetrization with respect to a part of indices} is formally given by 
\[
\sigma(i_{1},\cdots,i_{p}) u_{i_{1},\cdots,i_{m}} = \frac{1}{p!} \sum_{\pi \in \Pi_{p}} u_{i_{\pi(1)}\cdots i_{\pi(p)} i_{p+1} \cdots i_{m}},
\]
where $\Pi_{p}$ is the set of $p$-permutations. Let $\mathscr{D}'(\mathbb{R}^{n};S^{m})$ be the symmetric $m$ tensor distribution fields in $\mathbb{R}^{n}$. The \emph{inner derivative} or \emph{symmetrized derivative} is denoted as $\D : \mathscr{D}'(\Rb^n; S^m) \to \mathscr{D}'(\Rb^n; S^{m+1})$ given by 
\begin{equation}
(\D f)_{i_{1} \cdots i_{m+1}} = \sigma(i_{1} \cdots i_{m+1}) \frac{\partial f_{i_{1} \cdots i_{m}}}{\partial x_{i_{m+1}}} \equiv \frac{1}{(m+1)!} \sum_{\pi \in \Pi_{m+1}} \frac{\partial f_{i_{\pi(1)}\cdots i_{\pi(m)}}}{\partial x_{i_{\pi(m+1)}}} \label{eq:inner-product}
\end{equation}
The \emph{divergence} $\delta : \mathscr{D}'(\Rb^n; S^m) \to \mathscr{D}'(\Rb^n; S^{m-1})$ is defined by 
\begin{equation}\label{def_of_delta}
(\delta f)_{i_{1} \dots i_{m-1}} = \frac{\partial f_{i_{1} \dots i_{m}}}{\partial x^{i_{m}}}.
\end{equation}
The operators $\D$ and $-\delta$ are formally dual to each other with respect to $L^2$ inner product in the sense of 
\[
(\D u,v)_{\mathbb{R}^{n};S^{m+1}} = - (u,\delta v)_{\mathbb{R}^{n};S^{m}}
\]
for all $u \in C_{c}^{\infty}(\mathbb{R}^{n};S^{m})$ and $v \in C_{c}^{\infty}(\mathbb{R}^{n};S^{m+1})$, see \cite[(2.1.8)]{Sharafutdinov_Book}.
It is also interesting to mention that $\delta^{k+1} N^k f = 0$ and for each integer $0 \le r \le k$ that 
\begin{equation}
\begin{aligned}
\left( \delta^{r} N^{k} f \right)(x) &= \frac{k!}{(k-r)!} \int\limits_{\mathbb{S}^{n-1}} \langle x,\xi \rangle^{k-r} \xi^{\otimes (m-r)} \left( I^{k}f \right) (x - \langle x,\xi \rangle\xi , \xi) \, \D S_{\xi} \\
&= \frac{k!}{(k-r)!} \sum\limits_{\ell=0}^k \binom{k}{\ell} \int\limits_{\mathbb{S}^{n-1}} \langle x, \xi\rangle^{2k-r-\ell} \xi^{\otimes (m-r)} \left( J^{\ell} f \right)(x, \xi) \, \D S_{\xi}, 
\end{aligned} \label{eq:divergence-normal-operator}
\end{equation}
for all $f \in \mathcal{E}'(\mathbb{R}^{n};S^{m}) \cup \mathcal{S}(\mathbb{R}^{n};S^{m})$, see \cite[(2.14)]{DVS_UCP_2022}. 
Given any $m \in \mathbb{Z}_{\ge 0}$ and an integer $0 \le k \le m$, it is well-known that the momentum ray transforms (in Definition~{\rm \ref{def:MRT-schwartz}}) satisfy the following  property: 
\begin{equation}
I^{k}|_{S^{m}}(\D v) = - k I^{k-1}|_{S^{m-1}} (v) \quad \text{for all $v \in \Ec'(\mathbb{R}^{n};S^{m-1}) \cup \mathcal{S}(\mathbb{R}^{n};S^{m-1})$,} \label{eq:scaling-MRT}
\end{equation}
where $\D$ is given in \eqref{eq:inner-product}, see e.g. \cite[Section 2]{Krishnan2018}. By utilizing \eqref{eq:scaling-MRT}, we can easily show that the similar scaling property also holds for their corresponding normal operators. 

\begin{lemma}\label{lem:normal-operator-derivative-commute}
Given any $m \in \mathbb{Z}_{\ge 0}$ and integer $0 \le k \le m$, there holds 
\begin{equation*}
\delta N^{k}|_{S^{m}}(\D v) = -k^{2} N^{k-1}|_{S^{m-1}} (v) \quad \text{for all $v \in \Ec'(\mathbb{R}^{n};S^{m-1}) \cup \mathcal{S}(\mathbb{R}^{n};S^{m-1})$.}
\end{equation*}
\end{lemma}

\begin{proof}
Choosing $f = \D v$ and $r=1$ in \eqref{eq:divergence-normal-operator}, using \eqref{eq:scaling-MRT} one obtain
\[
\begin{aligned}
\left( \delta N^{k}|_{S^{m}} \D v \right)(x) &= k \int\limits_{\mathbb{S}^{n-1}} \langle x,\xi \rangle^{k-1} \xi^{\otimes (m-1)} \left( I^{k}|_{S^{m}} (\D v) \right) (x - \langle x,\xi \rangle\xi , \xi) \, \D S_{\xi} \\
&= -k^{2} \int\limits_{\mathbb{S}^{n-1}} \langle x,\xi \rangle^{k-1} \xi^{\otimes (m-1)} \left( I^{k-1}|_{S^{m-1}} (v) \right) (x - \langle x,\xi \rangle\xi , \xi) \, \D S_{\xi} \\
&= -k^{2} \left( N^{k-1}|_{S^{m-1}} (v) \right)(x),
\end{aligned}
\]
which immediately implies our lemma. 
\end{proof}
\addtocontents{toc}{\SkipTocEntry}
\subsection{Generalized Saint-Venant operator}
We now denote \tteal{$S^{m_{1}} \otimes S^{m_{2}}$} the set of $(m_{1}+m_{2})$-tensors symmetric with respect to  first $m_{1}$ indices and last $m_{2}$ indices. Accordingly, we introduce the generalized Saint-Venant operator as in \cite[(2.8)]{DVS_UCP_2022}, which is a generalization of curl on $\mathbb{R}^{n}$, see Remark~{\rm \ref{rem:curl-generalization}}. 

\begin{definition}\label{def:GSV}
For $m \in \mathbb{Z}_{\ge 0}$ and integer $0 \le k \le m$, the $k^{\rm th}$ \emph{generalized Saint-Venant operator} $W^{k} : \mathscr{D}'(\mathbb{R}^{n}; S^{m}) \to \mathscr{D}'(\mathbb{R}^{n}; S^{m-k} \otimes S^{m})$ is defined as
\begin{equation}
\begin{aligned}
& (W^{k} f)_{p_{1} \cdots p_{m-k} q_{1} \cdots q_{m-k}}^{i_{1} \cdots i_{k}} \\
& \quad \coloneqq \sigma(p_{1}, \cdots , p_{m-k}) \sigma(q_{1}, \cdots , q_{m-k}, i_{1} \dots i_{k}) \times \\
& \quad\qquad \times \sum_{\ell=0}^{m-k} (-1)^\ell \binom{m-k}{\ell} \frac{\partial^{m-k} f^{i_{1} \cdots i_k}_{p_1 \dots p_{m-k-\ell} q_1 \cdots q_{\ell}}}{\partial x^{p_{m-k-\ell+1}} \cdots \partial x^{p_{m-k}} \partial x^{q_{\ell+1}} \cdots \partial x^{q_{m-k}}}. 
\end{aligned}\label{GSV}
\end{equation}
\end{definition}
Since we work with the Euclidean metric, we will not distinguish between covariant tensors and contravariant tensors, i.e. there are no difference between upper and lower indices. We choose these notations just for own convenience while proving Proposition~3.1 in the next section.
When $k = 0$, the $0^{\rm th}$ generalized Saint-Venant operator $W^{0}$ reduces to the (classical) Saint-Venant operator on $S^{m}$ \cite[(2.8)]{Sharafutdinov_Book}. When $k=m$, we have $W^{m} = \mathbb{I}$, where 
$\mathbb{I}$ is the identity operator. 

\begin{remark}
[An equivalent definition]\label{eq:equivalence-SV-operator}
For each $m \in \mathbb{Z}_{\ge 0}$ and integer $0 \le k \le m$, we consider the operator \tteal{$R^{k} \equiv R^{k}|_{S^{m}} : \mathscr{D}'(\mathbb{R}^{n};S^{m}) \rightarrow \mathscr{D}'(\mathbb{R}^{n};T^{2m-k})$} given by 
\begin{equation}\label{GSV-equivalent}
(R^{k} f)_{p_{1} q_{1} \dots p_{m-k} q_{m-k}}^{i_{1} \dots i_{k}} := \alpha (p_{1} q_{1}) \dots \alpha (p_{m-k} q_{m-k}) \frac{\partial^{m-k} f^{i_{1} \dots i_{k}}_{p_{1} \dots p_{m-k}}}{\partial x^{q_1} \dots \partial x^{q_{m-k}}} 
\end{equation}
where the \emph{alternation of two indices} is defined as:
\[
\alpha(i_{1}i_{2}) u_{i_{1}i_{2}j_{1}\cdots j_{p}}:= \frac{1}{2} (u_{i_{1}i_{2}j_{1}\cdots j_{p}} - u_{i_{2}i_{1}j_{1}\cdots j_{p}}).
\]
In particular, one has 
\begin{equation}
\left( R^{k}|_{S^{m}} f \right)_{p_1 q_1 \dots p_{m-k} q_{m-k}}^{i_{1} \dots i_{k}} = \left( R^{0}|_{S^{m-k}} f^{i_{1}\cdots i_{k}} \right)_{p_{1}q_{1} \cdots p_{m-k}q_{m-k}}, \label{eq:reduction-GSV}
\end{equation}
see \cite[(2.6)]{DVS_UCP_2022}. For each $m \in \mathbb{Z}_{\ge 0}$ and integer $0 \le k \le m$, one can show that\footnote{We again remind readers about the typos in \cite[(2.4.6)(2.4.7)]{Sharafutdinov_Book}, see \cite[(2.6)(2.7)]{DVS_UCP_2022} for corrected statement.} 
\[
\begin{aligned}
&(W^{k}f)_{p_{1}\cdots p_{m-k} q_{1}\cdots q_{m-k}}^{i_{1}\cdots i_{k}} \\
&\quad = 2^{m-k} \sigma(q_{1} \cdots q_{m-k} i_{1} \cdots i_{k}) \sigma(p_{1} \cdots p_{m-k}) (R^{k}f)_{p_{1}q_{1} \cdots p_{m-k}q_{m-k}}^{i_{1}\cdots i_{k}},
\end{aligned}
\]
and
\[
\begin{aligned}
&(R^{k}f)_{p_{1}q_{1} \cdots p_{m-k}q_{m-k}}^{i_{1}\cdots i_{k}} \\
&\quad = \frac{1}{m-k+1} \binom{m}{k} \alpha(p_{1}q_{1}) \cdots \alpha(p_{m-k}q_{m-k}) (W^{k}f)_{p_{1}\cdots p_{m-k} q_{1}\cdots q_{m-k}}^{i_{1}\cdots i_{k}},
\end{aligned}
\]
see \cite[Lemma~4.1]{DVS_UCP_2022}. Hence, for each open set $U \subset \mathbb{R}^{n}$ and $f \in \mathscr{D}'(\mathbb{R}^{n};S^{m})$ there holds 
\begin{equation}
W^{k}f =0 \text{ in $U$} \iff R^{k}f =0 \text{ in $U$.} \label{eq:GSV-equivalence-UCP}
\end{equation}
Based on the above observation, we can slightly abuse the terminology by also referring \eqref{GSV-equivalent} the $k^{\rm th}$ generalized Saint-Venant operator. In view of the reduction formula \eqref{eq:reduction-GSV}, it is more convenient to work with $R^{k}$ rather than $W^{k}$. 
\end{remark}

\begin{remark}\label{rem:curl-generalization}
By identifying $\mathscr{D}'(\mathbb{R}^{n};S^{1}) \cong \left( (\mathscr{D}'(\mathbb{R}^{n})) \right)^{n}$, when $m=1$ and $k=0$, we see that 
\begin{equation}
\left (R^{0}|_{S^{1}}f \right)_{p,q} = \alpha(p,q) \frac{\partial f_{p}}{\partial x^{q}} = \frac{1}{2} \left( \frac{\partial f_{p}}{\partial x^{q}} - \frac{\partial f_{q}}{\partial x^{p}} \right) \equiv \frac{1}{\sqrt{2}} \left( {\rm curl}\,(f) \right)_{p,q} \label{eq:curl-Rn}
\end{equation}
for all $f \in \mathscr{D}'(\mathbb{R}^{n};S^{1}) \cong \left( \mathscr{D}'(\mathbb{R}^{n}) \right)^{n}$, see also Remark~{\rm \ref{rem:curl-curl}} for more details about \eqref{eq:curl-Rn}. 
\end{remark}

For $u \in S^{k}$, we denote by $\mathfrak{i}_{u} : S^{m} \rightarrow S^{m+k}$ the operator of symmetric multiplication by $u$ and by $\mathfrak{j}_{u} : S^{m+k} \rightarrow S^{m}$ the corresponding dual operator, 
 and defined as
\[
\begin{aligned}
\left( \mathfrak{i}_{u}v \right)_{i_{1} \cdots i_{m+k}} &:= \sigma(i_{1} \cdots i_{m+k}) u_{i_{1} \cdots i_{k}} v_{i_{k+1} \cdots i_{k+m}} \quad \text{for all $v \in S^{m}$}, \\
\tteal{\left( \mathfrak{j}_{u}w \right)_{i_{1} \cdots i_{m}}} &:= w_{i_{1} \cdots i_{m+k}} u^{i_{m+1} \cdots i_{m+k}} \quad \text{for all $w \in S^{m+k}$,}
\end{aligned}
\]
see \cite[(2.1.5)]{Sharafutdinov_Book}. Let $\mathfrak{e}^{k} \in S^{k}$ be the Euclidean metric tensor, given by 
\[
\mathfrak{e}_{i_{1} \cdots i_{k}}^{k} = \begin{cases}1 & \text{if $i_{1} = \cdots = i_{k}$,} \\ 0 & \text{otherwise,}\end{cases}
\]
and we write $\mathfrak{i}_{(k)} := \mathfrak{i}_{\mathfrak{e}^{k}}, \quad \mathfrak{j}_{(k)} := \mathfrak{j}_{\mathfrak{e}^{k}}$. We end this section by recalling \cite[Proposition~4.4]{DVS_UCP_2022}, which gives a connection between \tteal{the} normal operator of momentum ray transform and \tteal{the} generalized Saint-Venant operator. 

\begin{proposition}[{\cite[Proposition 4.4]{DVS_UCP_2022}}]  \label{MRT:Prop} 
Given $m \in \mathbb{Z}_{\ge 0}$ and integer $0 \le k \le m$, we consider $f \in \Ec'(\mathbb{R}^{n}; S^{m}) \cup \mathcal{S}(\mathbb{R}^{n};S^{m})$. If $N^{0}|_{S^{0}}$ be the operator given in \eqref{eq:normal_operator_S0}, then there holds
\begin{equation}
\begin{split}\label{key_equation_mrt}
& m! N^{0}|_{S^{0}} \left( (R^{0}|_{S^{m-k}} f^{i_{1} \cdots i_{k}})_{p_{1} q_{1} \cdots p_{m-k} q_{m-k}} \right) \\
& \quad = \sigma(i_{1} \dots i_{k}) \sum\limits_{r=0}^{k} (-1)^{r} \binom{k}{r} \frac{\partial^{r}}{\partial x^{i_1} \cdots \partial x^{i_r} }  (R^{k}|_{S^{m-r}} \left( G_{m-r}) \right)_{p_{1} q_{1} \cdots p_{m-k} q_{m-k}}^{i_{r+1} \cdots i_{k}},
\end{split}
\end{equation}
where $G_{m-r}$ is a symmetric $(m-r)$-tensor given by 
\[
\begin{aligned}
G_{m-r} = \sum\limits_{\ell=0}^{\lfloor \frac{m-r}{2} \rfloor} c_{\ell, m-r} \mathfrak{i}_{(2)}^{\ell} \mathfrak{j}_{(2)}^{\ell} \sum\limits_{p=0}^{r} (-1)^{r-p} \frac{1}{p!} \binom{r}{p}  \mathfrak{j}_{x^{\otimes (r-p)}} \delta^{p} N^{p} f,
\end{aligned}
\]
with coefficients 
\[
c_{\ell,s} = \left( \prod_{p=0}^{s-\ell-1}(n-1+2p) \right) \frac{(-1)^{\ell}s!}{2^{\ell} \ell ! (s-2\ell)!},
\]
where $N^{p}$ is the normal operator of the \tteal{$p^{\rm th}$} momentum ray transform given in \eqref{eq:normal_operator} and for each $\alpha \in \mathbb{R}$ the ``floor'' $\lfloor\alpha\rfloor$ denotes the largest integer with $\le \alpha$. 
\end{proposition}

\begin{remark}
\begin{subequations}
In particular when $m=1$ and $k=0$, \eqref{key_equation_mrt} reads 
\begin{equation}
N^{0}|_{S^{0}} \left( \left( R^{0}|_{S^{1}}f \right)_{ij} \right) = (n-1) \left( R^{0}|_{S^{1}} \left( N^{0}|_{S^{1}}f \right) \right)_{ij}\,\,\, \text{ $\forall\,\, f \in \Ec'(\mathbb{R}^{n};S^{1}) \cup \mathcal{S}(\mathbb{R}^{n};S^{1})$.} \label{eq:key-equation_mrt-special1}
\end{equation}
Plugging \eqref{eq:curl-Rn} into \eqref{eq:key-equation_mrt-special1}, we conclude
\begin{equation}
N^{0}|_{S^{0}} \left( \left( {\rm curl}\,(f) \right)_{ij} \right) = (n-1) \left( {\rm curl}\, \left( N^{0}|_{S^{1}}f \right) \right)_{ij} \quad \text{ $\forall\,\, f \in \left( \Ec'(\mathbb{R}^{n}) \right)^{n} \cup \left( \mathcal{S}(\mathbb{R}^{n}) \right)^{n}$.} \label{eq:key-equation_mrt-special2}
\end{equation}
\end{subequations}
\end{remark}

\section{A generalization of the curl-curl identity\label{sec:saint Venant}}

We shall prove the following useful property of \tteal{the} generalized Saint-Venant operator (Definition~{\rm \ref{def:GSV}}), which is new  based on our knowledge. As an application, we also prove a smoothing property of \tteal{the} generalized Saint-Venant operator in Lemma~{\rm \ref{lem:generalized_saint_venant_local}}. 

\begin{proposition} \label{prop:delta_R_relation}
Given any integer $m \ge 0$, the identity
\begin{equation}
\begin{aligned}
& \frac{\partial^{\ell}}{\partial x^{j_{1}} \cdots \partial x^{j_{\ell}}} \left( R^{0}|_{S^{m}}f \right)_{i_{1}j_{1}\cdots i_{\ell}j_{\ell} i_{\ell+1} j_{\ell+1} \cdots i_{m} j_{m}} \\
&\quad = \frac{1}{2^\ell}\sigma(i_{1} \cdots i_{\ell}) \sum\limits_{p=0}^{\ell} \,\binom{\ell}{p} (-1)^p \frac{\partial^{p}}{\partial x^{i_{1}} \cdots \partial x^{i_{p}}} \Delta^{\ell-p} \left( R^{0}|_{S^{m-\ell}} \left( (\delta^p f)^{i_{p+1} \cdots i_{\ell}} \right) \right)_{i_{\ell+1}j_{\ell+1}\cdots i_{m} j_{m}}
\end{aligned} \label{delta_R_relation} 
\end{equation}
holds true for all $f\in \mathscr{D}'(\mathbb{R}^{n}; S^m)$ and integer $0\le \ell \le m$.
\end{proposition}

\begin{remark} \label{rem:curl-curl}
\tteal{When $m=\ell=1$, \eqref{delta_R_relation} reduces to the following well-known curl-curl identity} 
\begin{equation}
\Delta f^{i} - \frac{\partial}{\partial x^{i}} (\delta f) = 2 \frac{\partial}{\partial x^{j}} (R^{0}|_{S^{1}}f)_{i,j} = \frac{\partial}{\partial x^{j}} \left( \frac{\partial f_{i}}{\partial x^{j}} - \frac{\partial f_{j}}{\partial x^{i}} \right) = - \left( {\rm curl}^{\intercal} \, {\rm curl} \, (f) \right)_{i}, \label{eq:curl-curl-formula}
\end{equation}
where $
\left( {\rm curl}^{\intercal}\, (g) \right)_{i} := \frac{1}{\sqrt{2}} \frac{\partial}{\partial x^{j}} (g_{i,j} - g_{j,i})$
is the \emph{formal transpose of curl}. In 3-dimensional case, one even can reduce \eqref{eq:curl-curl-formula} to 
\begin{equation}
\Delta f - \nabla (\nabla \cdot f) = - \nabla \times (\nabla \times f) \quad \text{for all $f \in \left( \mathscr{D}'(\mathbb{R}^{3}) \right)^{3}$,}
\end{equation}
where $\nabla \times \cdot : \mathscr{D}'(\mathbb{R}^{3}) \rightarrow \mathscr{D}'(\mathbb{R}^{3})$ is the usual curl on $\mathbb{R}^{3}$.  
\end{remark}

\begin{proof}[Proof of Proposition~{\rm \ref{prop:delta_R_relation}}]
The identity \eqref{delta_R_relation} is trivial when $\ell=0$, we only need to prove \eqref{delta_R_relation} for $\ell \ge 1$. 
\smallskip

\noindent \textbf{Step 1: Basic case.} 
We first prove \eqref{delta_R_relation} when $\ell=1$. Given any integer $m \ge 1$, by applying the divergence operator $\delta$ on the classical Saint-Venant operator $R^{0} : \mathscr{D}'(\mathbb{R}^{n}; S^{m}) \rightarrow \mathscr{D}'(\mathbb{R}^{n}; T^{2m})$, we obtain
\begin{equation}
\begin{aligned}
& 2 \frac{\partial}{\partial x^{j_{1}}} \left( R^{0}|_{S^{m}}f \right)_{i_{1}j_{1} i_{2}j_{2} \cdots i_{m}j_{m}} \\
& \quad = 2 \frac{\partial}{\partial x^{j_{1}}} \left( \alpha(i_{1} j_{1}) \cdots \alpha(i_{m} j_{m}) \frac{\partial^{m}}{\partial x^{j_{1}}\cdots \partial x^{j_{m}}} f_{i_{1} i_{2} \cdots i_{m}} \right) \\
& \quad =  \Delta \alpha(i_{2} j_{2}) \cdots \alpha(i_{m} j_{m})\frac{\partial^{m-1}}{\partial x^{j_{2}}\cdots \partial x^{j_{m}}} f_{i_{2} \cdots i_{m}}^{i_{1}} \\
& \quad \qquad - \frac{\partial}{\partial x^{i_{1}}} \alpha(i_{2} j_{2}) \cdots \alpha(i_{m} j_{m})\frac{\partial^{m-1}}{ \partial x^{j_{2}}\cdots \partial x^{j_{m}}} (\delta f)_{i_{2} \cdots i_{m}}  \\
& \quad = \left( \Delta R^{0}|_{S^{m-1}} f^{i_{1}} - \frac{\partial}{\partial x^{i_{1}}} R^{0}|_{S^{m-1}}(\delta f) \right)_{i_{2} j_{2} \cdots i_{m} j_{m}} 
\end{aligned} \label{eq:divergence-R-order1}
\end{equation}
for all $f\in \mathscr{D}'(\mathbb{R}^{n}; S^m)$. \smallskip

\noindent \textbf{Step 2: Induction on $\ell$.} Assume that there exists an integer $\ell \ge 1$ such that \eqref{delta_R_relation} holds true for all $m \ge \ell$. Taking the derivative $\frac{\partial}{\partial x^{j_{\ell+1}}}$ on the induction hypothesis gives 
\begin{equation}
\begin{aligned}
& \qquad 2^{\ell} \frac{\partial^{\ell+1}}{\partial x^{j_{1}} \cdots \partial x^{j_{\ell+1}}} \left( R^{0}|_{S^{m}}f \right)_{i_{1}j_{1}\cdots i_{\ell+1}j_{\ell+1} i_{\ell+2} j_{\ell+2} \cdots i_{m} j_{m}} \\
& = \sigma(i_{1} \cdots i_{\ell+1}) 2^{\ell} \frac{\partial^{\ell+1}}{\partial x^{j_{1}} \cdots \partial x^{j_{\ell+1}}} \left( R^{0}|_{S^{m}}f \right)_{i_{1}j_{1}\cdots i_{\ell+1}j_{\ell+1} i_{\ell+2} j_{\ell+2} \cdots i_{m} j_{m}} \\
& = \sigma(i_{1} \cdots i_{\ell+1}) \sum\limits_{p=0}^{\ell} \,\binom{\ell}{p} (-1)^p \frac{\partial^{p}}{\partial x^{i_{1}} \cdots \partial x^{i_{p}}} \Delta^{\ell-p} \frac{\partial}{\partial x^{j_{\ell+1}}} \left( R^{0}|_{S^{m-\ell}} \left( (\delta^p f)^{i_{p+1}\cdots i_{\ell}} \right) \right)_{i_{\ell+1}j_{\ell+1}\cdots i_{m} j_{m}}, 
\end{aligned} \label{eq:divergence-R-induction-ell-1}
\end{equation}
here we utilize the fact that $\sigma(i_{1} \cdots i_{\ell+1}) \sigma(i_{1} \cdots i_{\ell}) = \sigma(i_{1} \cdots i_{\ell+1})$. Similar to \eqref{eq:divergence-R-order1}  \tteal{after a re-indexing}, we have
\begin{align*}
 & (\delta^{p+1}f)_{i_{\ell+2}\cdots i_{m}}^{i_{p+1} \cdots i_{\ell}} = (\delta^{p+1}f)_{i_{\ell+2}\cdots i_{m}}^{i_{p+2} \cdots i_{\ell+1}} \quad \text{(because $\delta^{p+1}f \in C^{\infty}(\mathbb{R}^{n}; S^{m-(p+1)})$),}\\
   &\frac{\partial}{\partial x^{i_{\ell+1}}} R^{0}|_{S^{m-(\ell+1)}} \left (\delta^{p+1}f \right)^{i_{p+2} \cdots i_{\ell+1}} = \frac{\partial}{\partial x^{i_{p+1}}} R^{0}|_{S^{m-(\ell+1)}} \left (\delta^{p+1}f \right)^{i_{p+2} \cdots i_{\ell+1}}.
\end{align*}
This implies
\begin{equation}
\begin{aligned}
& 2 \frac{\partial}{\partial x^{j_{\ell+1}}} \left( R^{0}|_{S^{m-\ell}} \left( (\delta^p f)^{i_{p+1}\cdots i_{\ell}} \right) \right)_{i_{\ell+1}j_{\ell+1}\cdots i_{m} j_{m}} \\
& \quad = 2 \frac{\partial}{\partial x^{j_{\ell+1}}} \alpha(i_{\ell + 1}j_{\ell + 1}) \cdots \alpha(i_{m}j_{m}) \frac{\partial^{m-\ell}}{\partial x^{j_{\ell+1}} \cdots \partial x^{j_{m}}} \left( \delta^{p}f \right)_{i_{\ell+1} \cdots i_{m}}^{i_{p+1}\cdots i_{\ell}} \\
& \quad = \Delta \alpha(i_{\ell+2}j_{\ell+2}) \cdots \alpha(i_{m}j_{m}) \frac{\partial^{m-(\ell+1)}}{\partial x^{j_{\ell+2}} \cdots \partial x^{j_{m}}} (\delta^{p}f)_{i_{\ell+2} \cdots i_{m}}^{i_{p+1} \cdots i_{\ell} i_{\ell+1}} \\
& \quad \qquad - \frac{\partial}{\partial x^{i_{\ell+1}}} \alpha(i_{\ell+2}j_{\ell+2}) \cdots \alpha(i_{m}j_{m}) \frac{\partial^{m-(\ell+1)}}{\partial x^{j_{\ell+2}} \cdots \partial x^{j_{m}}} (\delta^{p+1}f)_{i_{\ell+2}\cdots i_{m}}^{i_{p+2} \cdots i_{\ell+1}} \\
& \quad = \left( \Delta R^{0}|_{S^{m-(\ell+1)}} \left (\delta^{p}f \right)^{i_{p+1} \cdots i_{\ell+1}} - \frac{\partial}{\partial x^{i_{p+1}}} R^{0}|_{S^{m-(\ell+1)}} \left (\delta^{p+1}f \right)^{i_{p+2} \cdots i_{\ell+1}} \right)_{i_{\ell+2}j_{\ell+2} \cdots i_{m} j_{m}}.
\end{aligned} \label{eq:divergence-R-induction-ell-2}
\end{equation}
The displayed relations between $3.4$ and $3.5$ can be deduced just from the definition of symmetric tensor fields. Since $ \delta^{p+1} f$ is a symmetric tensor field, we only need to count the indices. Note that,    $(\delta^{p+1}f)_{i_{\ell+2}\cdots i_{m}}^{i_{p+1} \cdots i_{\ell}}$, $i_{p+1}$ plays the role of $ i_{\ell+1}$.
We now combine \eqref{eq:divergence-R-induction-ell-1} and \eqref{eq:divergence-R-induction-ell-2} to obtain 
\begin{equation*}
\begin{aligned}
& \qquad 2^{\ell+1} \frac{\partial^{\ell+1}}{\partial x^{j_{1}} \cdots \partial x^{j_{\ell+1}}} \left( R^{0}|_{S^{m}}f \right)_{i_{1}j_{1}\cdots i_{\ell+1}j_{\ell+1} i_{\ell+2} j_{\ell+2} \cdots i_{m} j_{m}} \\ 
& = \sigma(i_{1} \cdots i_{\ell+1}) \sum_{p=0}^{\ell} \binom{\ell}{p} (-1)^{p} \frac{\partial^{p}}{\partial x^{i_{1}} \cdots \partial x^{i_{p}}} \Delta^{(\ell+1)-p} \left( R^{0}|_{S^{m-(\ell+1)}} \left (\delta^{p}f \right)^{i_{p+1} \cdots i_{\ell+1}} \right)_{i_{\ell+2}j_{\ell+2}\cdots i_{m}j_{m}} \\
& \quad - \sigma(i_{1} \cdots i_{\ell+1}) \sum_{p=0}^{\ell} \binom{\ell}{p} (-1)^{p} \frac{\partial^{p+1}}{\partial x^{i_{1}} \cdots \partial x^{i_{p+1}}} \Delta^{\ell-p} \left( R^{0}|_{S^{m-(\ell+1)}} \left (\delta^{p+1}f \right)^{i_{p+2} \cdots i_{\ell+1}} \right)_{i_{\ell+2}j_{\ell+2} \cdots i_{m} j_{m}} \\
& = \sigma(i_{1} \cdots i_{\ell+1}) \sum_{p=0}^{\ell} \binom{\ell}{p} (-1)^{p} \frac{\partial^{p}}{\partial x^{i_{1}} \cdots \partial x^{i_{p}}} \Delta^{(\ell+1)-p} \left( R^{0}|_{S^{m-(\ell+1)}} \left (\delta^{p}f \right)^{i_{p+1} \cdots i_{\ell+1}} \right)_{i_{\ell+2}j_{\ell+2}\cdots i_{m}j_{m}} \\
& \quad + \sigma(i_{1} \cdots i_{\ell+1}) \sum_{p=1}^{\ell+1} \binom{\ell}{p-1} (-1)^{p} \frac{\partial^{p}}{\partial x^{i_{1}} \cdots \partial x^{i_{p}}} \Delta^{(\ell+1)-p} \left( R^{0}|_{S^{m-(\ell+1)}} \left (\delta^{p}f \right)^{i_{p+1} \cdots i_{\ell+1}} \right)_{i_{\ell+2}j_{\ell+2} \cdots i_{m} j_{m}} \\
& = \sigma(i_{1} \cdots i_{\ell+1}) \sum_{p=0}^{\ell+1} \binom{\ell+1}{p} (-1)^{p} \frac{\partial^{p}}{\partial x^{i_{1}} \cdots \partial x^{i_{p}}} \Delta^{(\ell+1)-p} \left( R^{0}|_{S^{m-(\ell+1)}} \left (\delta^{p}f \right)^{i_{p+1} \cdots i_{\ell+1}} \right)_{i_{\ell+2}j_{\ell+2} \cdots i_{m} j_{m}}.
\end{aligned}
\end{equation*}
This finishes the induction step and completes the proof.
\end{proof}
When $\ell = m$, \eqref{delta_R_relation} in Proposition~{\rm \ref{prop:delta_R_relation}} gives 
\begin{equation}
\begin{aligned}
& \frac{\partial^{m}}{\partial x^{j_{1}} \cdots \partial x^{j_{m}}} \left( R^{0}|_{S^{m}}f \right)_{i_{1}j_{1}\cdots i_{m}j_{m}} \\
&\quad = \frac{1}{2^{m}}\sigma(i_{1} \cdots i_{m}) \sum\limits_{p=0}^{m} \,\binom{m}{p} (-1)^p \frac{\partial^{p}}{\partial x^{i_{1}} \cdots \partial x^{i_{p}}} \Delta^{m-p} \left( (\delta^p f)^{i_{p+1} \cdots i_{m}} \right). 
\end{aligned} \label{eq:delta_R_relation-1}
\end{equation}
As an application of \eqref{eq:delta_R_relation-1}, we \tteal{are} now able to prove the following smoothing property of \tteal{the} generalized Saint-Venant operator (Definition~{\rm \ref{def:GSV}}), \tteal{which is also true for $W^{k}$ and is seen by replacing $R^{k}$ with $W^{k}$ due to \eqref{eq:GSV-equivalence-UCP}}. 

\begin{lemma}\label{lem:generalized_saint_venant_local}
Given $m \in \mathbb{Z}_{\ge 0}$, let $U$ be an open set in $\mathbb{R}^{n}$ and $g \in \mathscr{D}'(\mathbb{R}^{n};S^{m})$. If there exists an integer $0 \le \mathsf{k} \le m$ such that
\begin{equation}
\text{$R^{\mathsf{k}}g = 0$ and $\delta^{\mathsf{k}+1}g=0$ in $U$,} \label{eq:assumption-generalized_saint_venant_local}
\end{equation}
then $g \in C^{\infty}(U ; S^{m})$. In addition when $g \in \mathscr{E}'(\mathbb{R}^{n};S^{m})$ and  \eqref{eq:assumption-generalized_saint_venant_local} holds for $U = \mathbb{R}^{n}$, then we conclude that $g \equiv 0$ in $\mathbb{R}^{n}$. 
\end{lemma}

\begin{proof}
From \eqref{GSV-equivalent}, we have 
\begin{equation}
\begin{aligned}
&(R^{k-1}|_{S^{m}} f)_{p_1 q_1 \dots p_{m-k+1} q_{m-k+1}}^{i_{1} \dots i_{k-1}} \\
& \quad = \alpha (p_{m-k+1} q_{m-k+1}) \frac{\partial}{\partial x^{q_{m-k+1}}} \left( \alpha (p_1 q_1) \dots \alpha (p_{m-k} q_{m-k}) \frac{\partial^{m-k} f_{p_{1} \dots p_{m-k+1}}^{i_1 \dots i_{k-1}}}{\partial x^{q_1} \dots \partial x^{q_{m-k}}} \right) \\
& \quad = \alpha (p_{m-k+1} q_{m-k+1}) \frac{\partial}{\partial x^{q_{m-k+1}}} \left( \alpha (p_1 q_1) \dots \alpha (p_{m-k} q_{m-k}) \frac{\partial^{m-k} f_{p_1 \dots p_{m-k}}^{i_{1} \dots i_{k-1} p_{m-k+1}}}{\partial x^{q_1} \dots \partial x^{q_{m-k}}} \right) \\
& \quad = \alpha (p_{m-k+1} q_{m-k+1}) \frac{\partial}{\partial x^{q_{m-k+1}}} (R^{k}|_{S^{m}} f)_{p_{1}q_{1} \cdots p_{m-k}q_{m-k}}^{i_{1}\cdots i_{k} p_{m-k+1}}
\end{aligned}
\end{equation}
for all $1 \le k \le m$ and $f \in \mathscr{D}'(\mathbb{R}^{n};S^{m})$. Therefore the first assumption in \eqref{eq:assumption-generalized_saint_venant_local} and \eqref{eq:reduction-GSV} implies 
\[
\left( R^{0}|_{S^{m-k}} g^{i_{1} \cdots i_{k}} \right)_{p_{1}q_{1} \cdots p_{m-k}q_{m-k}} = 0 \quad \text{in $U$} \quad \text{for all $0 \le k \le \mathsf{k}$,}
\]
therefore we have 
\begin{equation}
\left( R^{0}|_{S^{m-k}} (\delta^{k} g) \right)_{p_{1}q_{1} \cdots p_{m-k}q_{m-k}} = 0 \quad \text{in $U$} \quad \text{for all $0 \le k \le \mathsf{k}$.} \label{eq:assumption-generalized_saint_venant_local-consequence1}
\end{equation}
Choosing $f = \delta^{k}g$ in \eqref{eq:delta_R_relation-1}, from \eqref{eq:assumption-generalized_saint_venant_local-consequence1} and second assumption in \eqref{eq:assumption-generalized_saint_venant_local}, we see that 
\begin{equation}
\begin{aligned}
0 &= 2^{m} \frac{\partial^{m}}{\partial x^{j_{1}} \cdots \partial x^{j_{m}}} \left( R^{0}|_{S^{m-k}}(\delta^{k}g) \right)_{i_{1}j_{1} \cdots i_{m}j_{m}} \\
&= \sigma(i_{1} \cdots i_{m}) \sum_{p=0}^{\mathsf{k}-k} \binom{m}{p} (-1)^{p} \frac{\partial^{p}}{\partial x^{i_{1}} \cdots \partial x^{i_{p}}} \Delta^{m-p} \left( (\delta^{p+k}g)^{i_{p+1} \cdots i_{m}} \right) \quad \text{in $U$.}
\end{aligned} \label{eq:bootstrap-regularity-g-1}
\end{equation}
\begin{subequations}
By choosing $k = \mathsf{k}$ in \eqref{eq:bootstrap-regularity-g-1}, one see that $\Delta^{m} \left( (\delta^{\mathsf{k}} g)^{i_{1} \cdots i_{m}} \right) = 0$ in $U$. Therefore, by local elliptic regularity, one know that 
\begin{equation}
\delta^{\mathsf{k}}g \in C^{\infty}(U;S^{m-\mathsf{k}}). \label{eq:bootstrap-regularity-g-2a}
\end{equation}
In addition, when $g \in \Ec'(\mathbb{R}^{n};S^{m})$ and $U = \mathbb{R}^{n}$, we know that 
\begin{equation}
\delta^{\mathsf{k}}g \equiv 0 \quad \text{in $\mathbb{R}^{n}$}. \label{eq:bootstrap-regularity-g-2b}
\end{equation}
\end{subequations}
\begin{subequations}
On the other hand, for each $0 \le k < \mathsf{k}$, we can write \eqref{eq:bootstrap-regularity-g-1} as 
\[
\begin{aligned}
& \Delta^{m} \left( (\delta^{k}g) \right)^{i_{1}\cdots i_{m}} = \sigma(i_{1} \cdots i_{m}) \Delta^{m} \left( (\delta^{k}g) \right)^{i_{1}\cdots i_{m}} \\
& \quad = - \sigma(i_{1} \cdots i_{m}) \sum_{p=1}^{\mathsf{k}-k} \binom{m}{p} (-1)^{p} \frac{\partial^{p}}{\partial x^{i_{1}} \cdots \partial x^{i_{p}}} \Delta^{m-p} \left( (\delta^{p+k}g)^{i_{p+1} \cdots i_{m}} \right) \\
& \quad = \sigma(i_{1} \cdots i_{m}) \sum_{p=0}^{\mathsf{k}-k-1} \binom{m}{p+1} (-1)^{p} \frac{\partial^{p+1}}{\partial x^{i_{1}} \cdots \partial x^{i_{p+1}}} \Delta^{m-p-1} \left( (\delta^{p+k+1}g)^{i_{p+2} \cdots i_{m}} \right) 
\end{aligned}
\]
in $U$. Therefore, for each $0 \le k < \mathsf{k}$ we have the implication 
\begin{equation}
\delta^{k+1}g \in C^{\infty}(U;S^{m-k-1}) \quad \text{implies} \quad \delta^{k}g \in C^{\infty}(U;S^{m-k}). \label{eq:bootstrap-regularity-g-3a}
\end{equation}
In addition, when $g \in \Ec'(\mathbb{R}^{n};S^{m})$ and $U = \mathbb{R}^{n}$, for each $0 \le k < \mathsf{k}$ we have the implication 
\begin{equation}
\delta^{k+1}g \equiv 0 \text{ in $\mathbb{R}^{n}$} \quad \text{implies} \quad \delta^{k}g \equiv 0 \text{ in $\mathbb{R}^{n}$.} \label{eq:bootstrap-regularity-g-3b}
\end{equation}
\end{subequations}
Combining \eqref{eq:bootstrap-regularity-g-2a} and \eqref{eq:bootstrap-regularity-g-3a} for the general case, and combining \eqref{eq:bootstrap-regularity-g-2b} and \eqref{eq:bootstrap-regularity-g-3b} for the case when $g \in \Ec'(\mathbb{R}^{n};S^{m})$ and $U = \mathbb{R}^{n}$, we conclude our result. 
\end{proof}

\section{A solenoidal decomposition theorem\label{sec:Decomposition-solenoidal-potential}}

 The main theme of this section is to prove a generalized solenoidal potential decomposition theorem, which is also new according to our best knowledge.  

\begin{proposition}\label{prop:generalized_helmholtz decomposition}
Let $\Omega$ be a bounded smooth domain in $\mathbb{R}^{n}$. For each $f \in H^{\alpha}(\Omega;S^{m})$ with $\alpha \in \mathbb{Z}_{\ge 0}$, there exists a unique decomposition 
\begin{equation}
f = \tilde{f} + \D^{k} v \quad \text{in $\Omega$} \label{eq:solenoidal-decomposition}
\end{equation}
with $\tilde{f} \in H^{\alpha}(\Omega;S^{m})$ \tteal{satisfying $\delta^{k} \tilde{f}=0$} and $v \in H^{\alpha+k}(\Omega;S^{m-k}) \cap H_{0}^{k}(\Omega;S^{m-k})$. 
\end{proposition}

We shall borrow some ideas from \cite[Lemma~7.1]{SS_linearize_polyharmonic} (see also \cite[Theorem~1.5]{Dairbekov_Sharafutdinov} and \cite[Theorem~3.3.2]{Sharafutdinov_Book}) to \tteal{prove} Proposition~{\rm \ref{prop:generalized_helmholtz decomposition}}. Formally acting $\delta^{k}$ on \eqref{eq:solenoidal-decomposition}, one reach the equation $\delta^{k}f = \delta^{k} \D^{k} v$. This suggests us to prove the following lemma. 

\begin{lemma}\label{lem:elliptic-equation}
Let $\Omega$ be a bounded smooth domain in $\mathbb{R}^{n}\, (n\ge 2)$, and given $k \in \mathbb{N}$ and $m \in \mathbb{N}$. Given any $h \in H^{\ell}(\Omega;S^{m})$ with integer $\ell \ge -k$, there exists a unique $w \in H^{2k+\ell}(\Omega;S^{m}) \cap H_{0}^{k}(\Omega;S^{m})$ such that 
\begin{equation}
(-1)^{k}\delta^{k} \D^{k} w = h \quad \text{in $\Omega$.} \label{eq:elliptic-equation}
\end{equation}
\end{lemma}

\begin{proof}
Since the principal symbol of $\delta^{k}$ (resp. $\D^{k}$) is $\bfi^{k} \mathfrak{j}_{\xi^{\otimes k}}$ (resp. $\bfi^{k} \mathfrak{i}_{\xi^{\otimes k}}$), then the principal symbol of $(-1)^{k} \delta^{k} \D^{k}$ is $\mathfrak{j}_{\xi^{\otimes k}} \mathfrak{i}_{\xi^{\otimes k}}$. By using \cite[Lemma 3.3.3]{Sharafutdinov_Book}, one have 
\[
\mathfrak{j}_{\xi} \mathfrak{i}_{\xi} = \frac{1}{m+1}|\xi|^{2}\mathbb{I} + \frac{m}{m+1} \mathfrak{i}_{\xi} \mathfrak{j}_{\xi} > 0 \quad \text{in $S^{m}$},
\]
where $\mathbb{I}$ is the identity operator. Since $\mathfrak{j}_{\xi^{\otimes (\ell+1)}} \mathfrak{i}_{\xi^{\otimes (\ell+1)}} = \mathfrak{j}_{\xi} (\mathfrak{j}_{\xi^{\otimes \ell}} \mathfrak{i}_{\xi^{\otimes \ell}}) \mathfrak{i}_{\xi}$ for all $\ell \in \mathbb{N}$, then we know that $\mathfrak{j}_{\xi^{\otimes k}} \mathfrak{i}_{\xi^{\otimes k}} > 0$ in $S^{m}$. Therefore from \cite[Exercises~5.11.3 and 5.11.4]{Taylor_pde_1}, we know that the mapping 
\[
(-1)^{k} \delta^{k} \D^{k} : H^{2k+\ell}(\Omega) \cap H_{0}^{k}(\Omega) \rightarrow H^{\ell}(\Omega) \text{ is Fredholm of index zero.}
\]
In view of Fredholm theory, it remains to show the solution of \eqref{eq:elliptic-equation} is unique. In particular, if $w \in H_{0}^{k}(\Omega;S^{m})$ satisfies $\delta^{k} \D^{k} w=0$, integration by parts yields 
\[
\langle \delta^{k} \D^{k} w , w \rangle = \langle \D^{k} w, \D^{k} w \rangle = 0,
\]
which implies $\D^{k}w = 0$. \tteal{Finally, using \cite[Theorem~1.3]{Dairbekov_Sharafutdinov} as in the proof of \cite[Lemma~7.1]{SS_linearize_polyharmonic}, we conclude our lemma. }
\end{proof}

We are now ready to prove the main result of this section. 

\begin{proof}[Proof of Proposition~{\rm \ref{prop:generalized_helmholtz decomposition}}]
By choosing $h = (-1)^{k}\delta^{k}f \in H^{\alpha-k}(\Omega;S^{m-k})$ in Lemma~{\rm \ref{lem:elliptic-equation}}, there exists a unique $v \in H^{\alpha+k}(\Omega;S^{m-k}) \cap H_{0}^{k}(\Omega;S^{m-k})$ such that 
\[
\delta^{k} \D^{k} v = \delta^{k}f. 
\]
Clearly, if we define $\tilde{f} := f - \D^{k} v$, then we have $\delta^{k} \tilde{f}=0$. The uniqueness of the decomposition also easily follows from the uniqueness of solution in Lemma~{\rm \ref{lem:elliptic-equation}}. 
\end{proof}

\section{Main results}\label{sec:main results}

\subsection{Unique continuation property for momentum ray transform on tensors\label{subsec:UCP-ray-transform-tensors}}

We say that a tensor $g \in C^{\infty}(\Rb^n; S^m)$ \emph{vanishes at $x_{0}$ of infinite derivative order} if 
\[
\left. \frac{\partial^{\alpha}}{\partial x^{\alpha}}g_{i_{1}\cdots i_{m}} \right|_{x=x_{0}} = 0 \quad \text{for all multi-index $\alpha$.}
\]
In this section, we generalize \cite[Theorem~2.4]{DVS_UCP_2022} in the following theorem. 
\goldbach*

In particular the assumption \eqref{eq:vanishes-infinite-derivative-order} makes sense by the following lemma. 

\begin{lemma}
Let $m \in \mathbb{N}$ and an integer $0 \le k \le m$. Suppose that $f$ satisfies \eqref{eq:assumption-f-normal-well-defined}. 
If there exists an open set $U$ in $\mathbb{R}^{n}$ such that $R^{k}f=0$ in $U$, then $N^{p}f$ is smooth in $U$ for all $0 \le p \le k$. 
\end{lemma}

\begin{proof}
Fix $x_{0} \in U$, and let $B$ be a ball in $U$ containing $x_{0}$. Using Proposition~{\rm \ref{prop:generalized_helmholtz decomposition}} (with $\alpha=0$), one can decompose $f \in L^{2}(B)$ in $B$ as 
\[
f = \tilde{f} + \D^{k+1}v \quad \text{in $B$}
\]
where $\tilde{f} \in L^{2}(B;S^{m})$ satisfying $\delta^{k+1}\tilde{f}=0$ in $B$ and $v \in H_{0}^{k+1}(B;S^{m-k-1})$. Note that the zero extension $v\chi_{B}$ of $v$ is in $H^{k+1}(\mathbb{R}^{n};S^{m-k-1})$ and satisfies 
\begin{equation}
f = \tilde{f} + \D^{k+1}(v\chi_{B}) \quad \text{in $\mathbb{R}^{n}$}, \label{eq:decomposition-zero-extension}
\end{equation}
by extending $\tilde{f}=f$ outside $B$. 
By using \cite[Theorem~2.17.2]{Sharafutdinov_Book} (with $F = \D^{k+1}(v\chi_{B}) \in \Ec'(\mathbb{R}^{n};S^{m})$) and the equivalence in Remark~{\rm \ref{eq:equivalence-SV-operator}}, we have $R^{k}(\D^{k+1}(v\chi_{B})) = 0$. Therefore, \tteal{applying} $R^{k}$ on \eqref{eq:decomposition-zero-extension} \tteal{one sees that} 
\begin{equation}
R^{k}f = R^{k}\tilde{f} \quad \text{in $\mathbb{R}^{n}$,}
\end{equation}
therefore we conclude that $R^{k}\tilde{f}|_{B}=0$. Since $\delta^{k+1}\tilde{f}=0$ in $B$, Lemma~{\rm \ref{lem:generalized_saint_venant_local}} guarantees that $\tilde{f}$ is smooth in $B$.
Since \eqref{eq:assumption-f-normal-well-defined}, we know that $N^{p}f$ is \tteal{well-defined by \eqref{eq:extension-normal-L1} (when $0 \le k \le n-1$) and \eqref{eq:extension-normal-compactly-distribution} (when $n \le k \le m$)}. By using \eqref{eq:scaling-MRT}, for each $0 \le p\le k$, acting $N^{p}$ on \eqref{eq:decomposition-zero-extension} one see that 
\begin{equation}
N^{p}f = N^{p}\tilde{f} \quad \text{in $\mathbb{R}^{n}$,} \label{eq:Npf-Npftilde}
\end{equation}
which implies that $N^{p}f$ is smooth in $B$, and hence smooth near $x_{0}$. By arbitrariness of $x_{0} \in U$, we conclude our lemma. 
\end{proof}

With Proposition~{\rm \ref{MRT:Prop}} and a unique continuation \tteal{principle} of ray transform of scalar functions in \cite[Theorem~1.1]{Keijo_partial_function} at hand, we now able to prove Theorem~{\rm\ref{mrt_ucp}}.

\begin{proof}[Proof of Theorem~{\rm\ref{mrt_ucp}}]
Since $N^{p}f$ vanishes at $x_{0}$ of infinite derivative order, by using Proposition~\ref{MRT:Prop} we know that $N^{0}|_{S^{0}}\left((R^{0}|_{S^{m-k}}f^{i_{1} \dots i_{k}})_{p_{1} q_{1} \dots p_{m-k} q_{m-k}} \right)$ vanishes at $x_{0}$ of infinite derivative order as well. Since $R^{0}|_{S^{m-k}}(f^{i_1\cdots i_k})=R^{k}|_{S^{m}}f=0$ in $U$, using \cite[Theorem~1.1]{Keijo_partial_function} and the equivalence \eqref{eq:GSV-equivalence-UCP} we reach $R^{k}f=W^{k}f=0$ in $\mathbb{R}^{n}$. Then our result follows from \cite[Theorem~2.17.2]{Sharafutdinov_Book}. 
\end{proof}
\begin{remark}
  Following the ideas of \cite{partial_data_JK_JFAA}, it is possible to improve the Theorem \ref{mrt_ucp}  by replacing the assumption $R^kf=0$ in $U$ with $ P(D) R^kf=0$ in $U$, where $P(D)$ is  a constant coefficient differential operator.
\end{remark}

\subsection{Measurable unique continuation property of momentum ray transform\label{subsec:MUCP-MRT}}

We first recall the following unique continuation property from positive measure set (MUCP) for fractional Laplace operator from \cite{Ghosh_Salo_single_measurement} (or \cite{GR19unique}).

\begin{lemma}[{\cite[Proposition~5.1]{Ghosh_Salo_single_measurement}\label{lem:MUCP}}] 
Let $ n\ge 1$ and $ \Omega$ be an open set in $\mathbb{R}^{n}$. Let $q \in L^{\infty}(\Omega)$ and assume that $u \in H^{s}(\mathbb{R}^{n}))$ with $s \in [\frac{1}{4},1)$ satisfies 
\begin{align*}
((-\Delta)^s +q)u=0 \quad \mbox{in} \quad \Omega.
\end{align*}
If there exists a measurable set $ E\subset \Omega$ with positive measure such that $u=0$ in $E$, then $u \equiv 0$ in $\mathbb{R}^{n}$.
\end{lemma}

\begin{remark}
See also \cite[Theorem~4]{GR19unique} for more general results. We also refer to Lemma~{\rm \ref{lem:UCP-open}} in the next section for antilocality property for fractional elliptic operators, which implies the unique continuation property from nonempty open sets. 
\end{remark}

By utilizing the above lemma, we can prove some MUCP results for ray transforms. 
\measurable*


\begin{proof}[Proof of Theorem~{\rm \ref{thm:MUCP-ray-transform-function}}]
It is well-known that (see e.g.\ \cite[Exercise~12.10]{Ilmavirta21XrayNote}) there exists a constant $c_{n} \neq 0$ such that 
\begin{equation}
(-\Delta)^{\frac{1}{2}} N^{0}f = c_{n} f \quad\text{in $\mathbb{R}^{n}$.} \label{eq:well-known-normal-operator}
\end{equation}
Now $f|_{U}=0$ entails that $\left( (-\Delta)^{1/2} \left( N^{0}f \right) \right)|_{U}=0$. Since $\left. \left( N^{0}f \right) \right|_{E}=0$, by applying Lemma~{\rm \ref{lem:MUCP}} (with $s=\frac{1}{2}$, $u = N^{0}f$ and $q = 0$), this implies that $N^{0}f=0$ in $\mathbb{R}^{n}$, and our result follows from \eqref{eq:well-known-normal-operator}. 
\end{proof}


By using \eqref{eq:key-equation_mrt-special2}, we also can obtain an analogue for vector fields. 

\begin{theorem}\label{thm:MUCP-ray-transform-vector-field}
Suppose $U \subseteq \Rb^n$ be any non empty open set and $ n\ge 2$. Suppose that ${\rm curl}\,(N^{0}|_{S^{1}}f) \in H^{\frac{1}{2}}(\mathbb{R}^{n})$ {\rm (}see Section~{\rm \ref{subsec:assumption-regularity-MUCP}}{\rm )}. If $\left( {\rm curl} \, (f) \right)|_{U}=0$ and there exists a positive measure set $E$ in $U$ such that $\left. {\rm curl}\,(N^{0}|_{S^{1}}f) \right|_{E}=0$, then ${\rm curl}\,(f) \equiv 0$ in $\mathbb{R}^{n}$. 
\end{theorem}

\begin{remark}
We now consider the case when $n=2,3$. Let $U$ be an open set in $\mathbb{R}^{n}$ which is star-shaped with respect to some $x_{0} \in U$. If $f \in ( C^{1}(U) )^{n}$ and ${\rm curl}\,(f) = 0$ in $U$, then it is well-known that there exists a potential $p \in C^{2}(U)$ such that $f = \nabla p$ in $U$, which is an consequence of Poincar\'{e} lemma for de Rham cohomology groups. For any dimension $n\ge 2$, one can use  \cite[Theorem~2.17.2]{Sharafutdinov_Book} to conclude the same result.
\end{remark}


\begin{proof}[Proof of Theorem~{\rm \ref{thm:MUCP-ray-transform-vector-field}}]
Combining \eqref{eq:key-equation_mrt-special2} and \eqref{eq:well-known-normal-operator}, one see that there exists a constant $c_{n}' \neq 0$ such that 
\begin{equation}
(-\Delta)^{\frac{1}{2}} \left( {\rm curl}\, (N^{0}|_{S^{1}}f) \right) = c_{n}' {\rm curl}\,(f) \quad \text{in $\mathbb{R}^{n}$.} \label{eq:well-known-operator-curl}
\end{equation}
Now $\left. {\rm curl} \, (f) \right|_{U}=0$ entails that $\left. (-\Delta)^{\frac{1}{2}} \left( {\rm curl}\, (N^{0}|_{S^{1}}f) \right) \right|_{U}=0$. Since $\left. {\rm curl}\,(N^{0}|_{S^{1}}f) \right|_{E}=0$, by applying Lemma~{\rm \ref{lem:MUCP}} (with $s=\frac{1}{2}$, $u = {\rm curl}\,(N^{0}|_{S^{1}}f)$ and $q=0$), this implies that ${\rm curl}\,(N^{0}|_{S^{1}}f) = 0$ in $\mathbb{R}^{n}$. Combining this with \eqref{eq:well-known-operator-curl}, we conclude our theorem.  
\end{proof}

\subsection{Generalization of momentum ray transform: Fractional momentum ray transform}\label{subsec:ray-transform-elliptic-operator}

The main theme of this section is to generalize the momentum ray transform to a more general weighted ray transform, and explain the connection with fractional elliptic operators. 

Let $n \ge 2$ be an integer. For each real number $0 < s < \frac{n}{2}$, we consider the \tteal{fractional momentum} ray transform $\mathcal{X}_{s} : \mathcal{S}(\mathbb{R}^{n}) \rightarrow C^{\infty}(\mathbb{R}^{n} \times \mathbb{S}^{n-1})$ defined by 
\begin{equation}
\big( \mathcal{X}_{s}f \big) (x,\xi) := \big( I^{2s-1} f \big) (x,\xi) \equiv \int\limits_{0}^{\infty} \tau^{2s-1} f(x + \tau \xi)\, \D \tau \quad \text{for all } f \in \Sc(\mathbb{R}^{n}) \label{eq:X-ray}
\end{equation}
and for all $(x,\xi) \in \mathbb{R}^{n} \times \mathbb{S}^{n-1}$. Rather than the normal operator $N^{2s-1}$, here we alternative consider the average $\mathcal{A}_{s} : \mathcal{S}(\mathbb{R}^{n}) \rightarrow C^{\infty}(\mathbb{R}^{n})$ over the sphere $\mathbb{S}^{n-1}$ \tteal{is defined by}  
\begin{equation}\label{def_A_s}
    \big( \mathcal{A}_{s} f \big) (x) := c(n,-s) \int\limits_{\mathbb{S}^{n-1}} \big( \mathcal{X}_{s}f \big) (x,\xi) \, \D S_{\xi},
\end{equation}
where $c(n,s) := \frac{2^{2s} \Gamma(\frac{n+2s}{2})}{\pi^{n/2}|\Gamma(-s)|}$. We compute that 
\begin{equation*}
\begin{aligned}
(\mathcal{A}_{s}f)(x) & = c(n,-s) \int\limits_{\mathbb{S}^{n-1}} \int\limits_{0}^{\infty} \tau^{2s-1} f(x + \tau \xi)\, \D \tau \, \D S_{\xi} \\ 
& = c(n,-s) \int\limits_{\mathbb{S}^{n-1}} \int\limits_{0}^{\infty} \tau^{2s-n} f(x + \tau \xi) \, \big( \tau^{n-1} \D \tau \, \D S_{\xi} \big) \\
& = c(n,-s) \int\limits_{\mathbb{R}^{n}} |y|^{2s-n}f(x+y)\,\D y,
\end{aligned} 
\end{equation*}
In other words, we observe that $\mathcal{A}_{s} : \mathcal{S}(\mathbb{R}^{n}) \rightarrow C^{\infty}(\mathbb{R}^{n})$ is simply the Riesz potential, and in particular from \cite[Theorem~5]{Sti19FractionalLaplacian} for each $f \in \mathcal{S}(\mathbb{R}^{n})$ we have  
\begin{equation}
(-\Delta)^{-s}f(x) = \frac{1}{\Gamma(s)} \int\limits_{0}^{\infty} e^{t\Delta} f(x) \, \frac{\D t}{t^{1-s}} = (\mathcal{A}_{s}f)(x) \quad \text{for all $x \in \mathbb{R}^{n}$,} \label{eq:semigroup-definition-negative-Laplacian}
\end{equation}
where $(-\Delta)^{-s}$ is the negative power of the Laplacian and $\{e^{t\Delta}\}_{t \ge 0}$ is the classical heat diffusion semigroup, see also \cite[Definition~2.3]{Kwa17FractionalEquivalent}. This corresponds to the numerical identity 
\begin{equation}
\lambda^{-s} = \frac{1}{\Gamma(s)} \int\limits_{0}^{\infty} e^{-t\lambda} \, \frac{\D t}{t^{1-s}} \quad \text{for all $\lambda>0$,} \label{eq:numerical-identity}
\end{equation}
see e.g. \cite[Section~2.7]{Kwa17FractionalEquivalent}. For each $p \in [1,\frac{n}{2s})$, by utilizing weak Young's inequality, one can show that 
\begin{equation}
\mathcal{A}_{s}:L^{p}(\mathbb{R}^{n}) \rightarrow L^{q}(\mathbb{R}^{n}) \text{ is bounded provided } \frac{1}{q} = \frac{1}{p} - \frac{2s}{n}, \label{eq:Riesz-domain-range}
\end{equation}
see also e.g. \cite[Section~2.7]{Kwa17FractionalEquivalent}. Using convolution theorem we observe that $\partial^{\alpha}(\mathcal{A}_{s}f)=\mathcal{A}_{s}(\partial^{\alpha}f)$, then in particular 
\begin{equation}
\mathcal{A}_{s}:W^{m,p}(\mathbb{R}^{n}) \rightarrow W^{m,q}(\mathbb{R}^{n}) \text{ is bounded provided } \frac{1}{q} = \frac{1}{p} - \frac{2s}{n}, \label{eq:Riesz-domain-range1}
\end{equation}
for all non-negative integer $m$. We are interested in the following particular case: 
\begin{lemma} \label{lem:averaging-operator-fractional-Laplacian}
Let $n \ge 2$ be an integer and let $0 < s \le \frac{n}{4}$, then for each non-negative integer $m$ we know that 
\[
\mathcal{A}_{s} : W^{m,\frac{2n}{n+4s}}(\mathbb{R}^{n}) \rightarrow W^{m,2}(\mathbb{R}^{n}) \equiv H^{m}(\mathbb{R}^{n}) \text{ is bounded,}
\]
and $(-\Delta)^{s}\mathcal{A}_{s}f=f$ for all $f \in W^{m,\frac{2n}{n+4s}}(\mathbb{R}^{n})$. Here, $(-\Delta)^{s}$ is the usual Fourier fractional Laplacian\footnote{See also e.g. \cite[Definition~1]{Sti19FractionalLaplacian} for \tteal{the} definition of \tteal{the} fractional Laplacian in some suitable distribution sense}. 
\end{lemma}

\begin{remark}[Critical case]
When $s = \frac{n}{2}$, we define the weighted ray transform $\mathcal{X}_{\frac{n}{2}} : \mathcal{S}(\mathbb{R}^{n}) \rightarrow C^{\infty}(\mathbb{R}^{n} \times \mathbb{S}^{n-1})$ by 
\begin{equation}
\big( \mathcal{X}_{\frac{n}{2}} f \big)(x,\xi) := \int_{0}^{\infty} (-2 \log \tau - \gamma_{\rm EM}) f(x + \tau \xi) \, \D \tau \quad \text{for all $(x,\xi) \in \mathbb{R}^{n} \times \mathbb{S}^{n-1}$,}
\end{equation}
where the Euler-Mascheroni constant $\gamma_{\rm EM}$ is given by 
$\gamma_{\rm EM} := - \int_{0}^{\infty} e^{-t} \log t \, \D t \approx 0.577215$. \tteal{The average operator} $\mathcal{A}_{\frac{n}{2}} : \mathcal{S}(\mathbb{R}^{n}) \rightarrow C^{\infty}(\mathbb{R}^{n})$ is then given by 
\[
\big( \mathcal{A}_{\frac{n}{2}} f \big) (x) := \frac{1}{\Gamma(\frac{n}{2})(4\pi)^{\frac{n}{2}}} \int\limits_{\mathbb{S}^{n-1}} \big( \mathcal{X}_{\frac{n}{2}} f \big) (x,\xi) \, \D S_{\xi}.
\]
We compute that
\begin{align*}
(\mathcal{A}_{\frac{n}{2}}f)(x) & = \frac{1}{\Gamma(\frac{n}{2})(4\pi)^{\frac{n}{2}}} \int\limits_{\mathbb{S}^{n-1}} \int\limits_{0}^{\infty} (-2 \log \tau - \gamma_{\rm EM}) f(x + \tau \xi)\, \D \tau \, \D S_{\xi} \\ 
& = \frac{1}{\Gamma(\frac{n}{2})(4\pi)^{\frac{n}{2}}} \int\limits_{\mathbb{S}^{n-1}} \int\limits_{0}^{\infty} (-2 \log \tau - \gamma_{\rm EM}) f(x + \tau \xi) \, \big( \tau^{n-1} \D \tau \, \D S_{\xi} \big) \\
& = \frac{1}{\Gamma(\frac{n}{2})(4\pi)^{\frac{n}{2}}} \int\limits_{\mathbb{R}^{n}} (-2 \log |y| - \gamma_{\rm EM})f(x+y)\,\D y.
\end{align*}
Since $|\xi|^{-n}$ is not a tempered distribution, then it is interesting to mention that \cite[Theorem~5]{Sti19FractionalLaplacian} showed that 
\[
((-\Delta)^{-\frac{n}{2}}f)(x) = (\mathcal{A}_{\frac{n}{2}}f)(x) \quad \text{for all $x \in \mathbb{R}^{n}$}
\]
for all $f \in \mathcal{S}(\mathbb{R}^{n})$ with $\int\limits_{\mathbb{R}^{n}}f(x)\, \D x = 0$. 
\end{remark}

For $0 < s < \frac{n}{2}$, we now express the weighted ray transform \eqref{eq:X-ray} in terms of (Gauss-Weierstrass) heat kernel 
\[
k_{t}(x,y) := \frac{1}{(4\pi t)^{\frac{n}{2}}} e^{-\frac{|x-y|^{2}}{4t}}.
\]
\tteal{It is known that}  
\[
(e^{t\Delta}f)(x) = \int\limits_{\mathbb{R}^{n}} k_{t}(x,y)f(y) \, \D y \quad \text{for all $x \in \mathbb{R}^{n}$} 
\]
and $(e^{t\Delta}f)^{\wedge}(x) = e^{-t|\xi|^{2}} \hat{f}(\xi)$ for all $f \in \mathcal{S}(\mathbb{R}^{n})$. In addition, from \cite[Section~2.7]{Kwa17FractionalEquivalent} we have 
\[
\frac{1}{\Gamma(s)} \int\limits_{0}^{\infty} k_{t}(x,y) \, \frac{\D t}{t^{1-s}} = c(n,-s) |x-y|^{2s-n}.
\]
Accordingly, we observe that 
\begin{equation*}
|x-y|^{2s-1} = \frac{1}{\Gamma(s) c(n,-s)} |x-y|^{n-1} \int\limits_{0}^{\infty} k_{t}(x,y) \, \frac{\D t}{t^{1-s}} \quad \mbox{for all $x \ne y \in \mathbb{R}^{n}$.} 
\end{equation*}
Therefore the X-ray transform \eqref{eq:X-ray} and the corresponding averaging operator \tteal{can be rephrased as} 
\begin{equation}
\begin{aligned}
(\mathcal{X}_{s}f)(x,\xi) &= \frac{1}{\Gamma(s) c(n,-s)} \int\limits_{0}^{\infty} \tau^{n-1} \bigg( \int\limits_{0}^{\infty} k_{t}(x,x+\tau\xi) \, \frac{\D t}{t^{1-s}} \bigg) f(x + \tau\xi) \, \D \tau \\
\left( (-\Delta)^{-s}f \right)(x) &\equiv (\mathcal{A}_{s}f)(x) = \frac{1}{\Gamma(s)} \int_{\mathbb{R}^{n}} \bigg( \int\limits_{0}^{\infty} k_{t}(x,y) \, \frac{\D t}{t^{1-s}} \bigg) f(y) \,\D y
\end{aligned} \label{eq:X-ray-heat-kernel}
\end{equation}
for all $f \in \mathcal{S}(\mathbb{R}^{n})$. Based on the above observation, it is possible \tteal{to define} negative power of elliptic operators with some suitable domain, see Appendix~{\rm \ref{appen:Negative-power-elliptic}} for more details. 
\begin{remark}
We have defined fractional MRT in \eqref{eq:X-ray} by taking integration over half lines. However one can define fractional MRT by taking integration over whole lines in the following way:
\begin{align*}
    J^sf(x,\xi):= \int_{\Rb} |t|^s \, f(x+t\xi)\, \D t  \quad \mbox{for $f\in \Sc (\Rn)$ and $-1<s<1$}. 
\end{align*}
Note that, the weight function we have considered here is not smooth at the origin. But after taking  the average over $ \Sn$ we see that 
$A_s=\int_{\Sn} J^s(x,\xi)\, \D S(\xi) $ is \tteal{same as $ \mathcal{A}_{\frac{s+1}{2}}$ up to a constant}. Using this we can study the UCP of $A_s$ as well as  $J^s$.
\end{remark}

\subsection{Antilocality property of the weighted ray transform\label{subsec:antilocality}}

Using \cite[Proposition~1.9]{GR19unique} (which is valid for some class of general elliptic operators) and the smoothing argument as in the proof of \cite[Theorem~1.2]{GSU-fractional}, we have the following lemma. 

\begin{lemma} \label{lem:UCP-open}
Let $n \ge 1$ be an integer, $s > 0$ with $s \notin \mathbb{Z}$. Let $u \in H^{r}(\mathbb{R}^{n})$ for some $r \in \mathbb{R}$. If 
\[
u = (-\Delta)^{s}u = 0 \quad \text{in some open set in $\mathbb{R}^{n}$,}
\]
then $u \equiv 0$ in $\mathbb{R}^{n}$. 
\end{lemma}

Based on Lemma~{\rm \ref{lem:averaging-operator-fractional-Laplacian}}, we can obtain the antilocality property of the weighted ray transform $\mathcal{A}_{s}$ using the antilocality property of the fractional Laplacian. 
\fractional*

\begin{remark}
If $\big( \mathcal{X}_{s}f \big)(x,\xi) = 0$ for all $x \in U$ and $\xi \in \mathbb{S}^{n-1}$, then $\mathcal{A}_{s}f = 0$ in $U$. 
\end{remark}

\begin{proof}[Proof of Theorem~{\rm \ref{thm:UCP-open-Xray}}]
Using Lemma~{\rm \ref{lem:averaging-operator-fractional-Laplacian}}, we know that 
\[
g := \mathcal{A}_{s}f \in L^{2}(\mathbb{R}^{n}), \quad g = (-\Delta)^{s}g = 0 \text{ in $U$.}
\]
From Lemma~{\rm \ref{lem:UCP-open}} we know that $g \equiv 0$ in $\mathbb{R}^{n}$. Consequently, again using Lemma~{\rm \ref{lem:averaging-operator-fractional-Laplacian}} we conclude $f \equiv (-\Delta)^{s}g \equiv 0$ in $\mathbb{R}^{n}$. 
\end{proof}

We also have the following interesting observation. 

\begin{lemma}[A support theorem for $\mathcal{A}_{s}$]\label{thm:support_thm_A_s}
Let $n \ge 2$ be an integer, $0 < s < 1$ and let $f \in W^{1,p}(\mathbb{R}^{n})$ for some $\max \{1,\frac{2n}{n+4s} \} \le p < \frac{n}{2s}$. If there exists a nonempty bounded Lipschitz domain $\Omega$ such that 
\[
\mathcal{A}_{s}f=0 \text{ in $\mathbb{R}^{n}\setminus\overline{\Omega}$},\quad f=0 \text{ in $\Omega$,}
\]
then $f \equiv 0$ in $\mathbb{R}^{n}$. 
\end{lemma}

\begin{proof}
Using \eqref{eq:Riesz-domain-range1}, we know that $g := \mathcal{A}_{s}f \in W^{1,p}(\mathbb{R}^{n})$ for some $p \ge 2$. Since $g$ has compact support and $0 < s < 1$, then $g \in H^{1}(\mathbb{R}^{n}) \subset H^{s}(\mathbb{R}^{n})$. From \eqref{eq:semigroup-definition-negative-Laplacian} we have 
\[
(-\Delta)^{s} g = 0 \text{ in $\Omega$} ,\quad g=0 \text{ in $\mathbb{R}^{n}\setminus\overline{\Omega}$.}
\]
Since 0 is not an eigenvalue of $(-\Delta)^{s}$ (a consequence of \cite[Proposition~3.3]{LL19GlobalUniquenessSemilinear} or \cite[Corollary~5.2]{RosOton16DirichletProblem}), then we know that $g \equiv 0$ in $\mathbb{R}^{n}$. Consequently, again using Lemma~{\rm \ref{lem:averaging-operator-fractional-Laplacian}} we conclude $f \equiv (-\Delta)^{s}g \equiv 0$ in $\mathbb{R}^{n}$. 
\end{proof}

\subsection{Antilocality property of the cone transform\label{subsec:antilocaloty-cone-transform}}

This weighted ray transform is related to the cone transform. Cone transform appears in different imaging approaches, most notably in the modeling of data provided by the so-called Compton camera, which has unique applications in domains such as medical and industrial imaging, homeland security, and gamma ray astronomy, see \cite{Kuchment_IPI}. Let $\mathfrak{C}(u,\beta,\psi)$ be the cone with vertex $u \in \mathbb{R}^{n}$, central axis $\beta\in \mathbb{S}^{n-1}$ and opening angle $\psi\in (0,\pi)$, that is, 
\[
\mathfrak{C}(u,\beta,\psi) := \begin{Bmatrix}\begin{array}{l|l} x \in \mathbb{R}^{n} & (x-u)\cdot \beta = |x-u| \cos \psi \end{array}\end{Bmatrix}.
\]
Following \cite{Kuchment_IPI}, we next introduce the  weighted cone transform.
\begin{definition}
For each $-1 < k < n-1$ and $f\in \mathcal{S}(\Rn)$, the $ k$-th weighted cone transform is defined as:
\begin{align*}
(\mathcal{C}^{k}f) (u,\beta,\psi)=\int\limits_{\mathfrak{C}(u,\beta,\psi)} f(x) \, |x-u|^{k-n+2}\, \D S_{x},
\end{align*}
where $\D S$ is the surface measure on the cone $\mathfrak{C}(u,\beta,\psi)$. 
\end{definition}

It is worth-mentioning that the cone transform is related to the momentum ray transform $I^{k} \equiv \mathcal{X}_{\frac{k+1}{2}}$, precisely, 
\begin{equation}
\int\limits_{0}^{\pi} (\mathcal{C}^{k}f)(u,\beta,\psi) h(\cos(\psi)) \,\D \psi = \int\limits_{\mathbb{S}^{n-1}} (I^{k} f)(u,\sigma) h(\sigma \cdot \beta) \, \D \sigma \quad \text{for all $f \in \mathcal{S}(\mathbb{R}^{n})$} \label{eq:cone-transform-ray-transform}
\end{equation}
for all distribution $h \in \mathcal{D}'(\mathbb{R}^{1})$ which are regular near $t=\pm 1$, see \cite[(19)]{Kuchment_IPI}. By writing $s = \frac{k+1}{2} \in (0,\frac{n}{2})$ and choosing $h \equiv c(n,-s)$ in \eqref{eq:cone-transform-ray-transform} yields 
\begin{equation}
c(n,-s)\int\limits_{0}^{\pi} (\mathcal{C}^{2s-1}f)(u,\beta,\psi) \, \D \psi = c(n,-s)\int\limits_{\mathbb{S}^{n-1}} (\mathcal{X}_{s}f)(u,\sigma) \, \D \sigma = (\mathcal{A}_{s}f)(u) \quad \text{for all $u \in \mathbb{R}^{n}$}
\end{equation}
for all $f \in \mathcal{S}(\mathbb{R}^{n})$. Therefore, as a corollary of Theorem~{\rm \ref{thm:UCP-open-Xray}}, we immediately obtain the antilocality property for cone transform. 

\begin{corollary}\label{cor:ucp_for_cone_transform}
Let $n \ge 2$ be an integer, $0 < s \le \frac{n}{4}$ with $s \neq \mathbb{Z}$ and let $f \in L^{\frac{2n}{n+4s}}(\mathbb{R}^{n})$. Suppose that there exists $\beta \in \mathbb{S}^{n-1}$ and a non-empty open set $U$ in $\mathbb{R}^{n}$ such that $f=0$ in $U$ and 
\[
(\mathcal{C}^{2s-1}f)(u,\beta,\psi) = 0 \quad \text{for all $(u,\psi)\in U\times (0,\pi)$,}
\]
then $f \equiv 0$ in $\mathbb{R}^{n}$. 
\end{corollary}

\appendix

\section{\label{appen:Negative-power-elliptic}Negative power of elliptic operator}

The main theme of this appendix is to explain the connection between some weighted x-ray transform with negative power of some class of elliptic operators on some natural domain. 
Let $A(x) = (a_{ij}(x))_{i,j=1}^{n} \in (C^{\infty}(\mathbb{R}^{n}))^{n \times n}$ such that $a_{ij}=a_{ji}$ for all $i,j=1,\cdots,n$.   Suppose that $A$ satisfies the following ellipticity condition: there exists a constant $0 < c < 1$ such that 
\begin{equation}
c |\xi|^{2} \le \xi \cdot A(x) \xi \le c^{-1} |\xi|^{2} \quad \text{for all } x \in \mathbb{R}^{n}. \label{eq:ellipticity}
\end{equation}
It is known that the operator $-\nabla \cdot A \nabla$ with domain $H^{2}(\mathbb{R}^{n})$ is the maximal extension such that it is self-adjoint\footnote{The operator $-\nabla \cdot A \nabla$ is not self-adjoint on the domain $C_{c}^{\infty}(\mathbb{R}^{n})$.} and densly defined in $L^{2}(\mathbb{R}^{n})$, see \cite[Theorem~4.6]{Gri09HeatKernelManifold}. Using the spectral theorem for self-adjoint operator in a real Hilbert space as in \cite[Appendix~A.5.4]{Gri09HeatKernelManifold}, there exists a unique \emph{spectral resolution} (also known as \emph{resolution of identity}) $\{E_{\lambda}\}$ in $L^{2}(\mathbb{R}^{n})$ corresponding to $-\nabla \cdot A \nabla$ such that 
\[
\left( - \nabla \cdot A \nabla f,g \right)_{L^{2}(\mathbb{R}^{n})} = \int_{0}^{\infty} \lambda \, \D (E_{\lambda}f,g)_{L^{2}(\mathbb{R}^{n})}
\]
for all $f \in H^{2}(\mathbb{R}^{n})$ and $g \in L^{2}(\mathbb{R}^{n})$. The detailed properties of spectral resolution can be found in \cite[Appendix~A.5.3]{Gri09HeatKernelManifold}, here we only state some of them. Let $\lambda \mapsto \varphi(\lambda)$ be a Borel function on $\mathbb{R}$. By utilizing Riesz representation theorem (see also \cite[(A.38)]{Gri09HeatKernelManifold}), one can define the self-adjoint operator $\varphi(-\nabla \cdot A \nabla)$ by 
\begin{equation}
\left( \varphi(-\nabla \cdot A \nabla)f ,g \right)_{L^{2}(\mathbb{R}^{n})} = \int\limits_{0}^{\infty} \varphi(\lambda) \, \D (E_{\lambda}f,g)_{L^{2}(\mathbb{R}^{n})} \equiv \bigg( \int\limits_{0}^{\infty} \varphi(\lambda) \, \D E_{\lambda}f ,g \bigg)_{L^{2}(\mathbb{R}^{n})} \label{eq:Borel-operator}
\end{equation}
for all $g \in L^{2}(\mathbb{R}^{n})$ and 
\begin{equation}
f \in {\rm dom}\, \left( \varphi(-\nabla \cdot A \nabla) \right) := \begin{Bmatrix}\begin{array}{l|l} f \in L^{2}(\mathbb{R}^{n}) & \displaystyle{\int\limits_{0}^{\infty}} |\varphi(\lambda)|^{2} \, \D \| E_{\lambda} f \|_{L^{2}(\mathbb{R}^{n})}^{2} < \infty \end{array} \end{Bmatrix}. \label{eq:domain-definition}
\end{equation}
One may define $(-\nabla\cdot A\nabla)^{-s}$ using the mapping $\varphi(\lambda) = \lambda^{-s}$. But however in general $\nabla\cdot A\nabla:H^{2}(\mathbb{R}^{n}) \rightarrow L^{2}(\mathbb{R}^{n})$ is not injective, therefore the domain ${\rm dom}\,((-\nabla\cdot A\nabla)^{-s})$ is somehow artificial. The main theme of this appendix is to define $(-\nabla\cdot A\nabla)^{-s}$ on a suitable domain. 

By using \cite[(A.32)]{Gri09HeatKernelManifold}, we know that ${\rm dom}\,(e^{t\nabla \cdot A \nabla}) = L^{2}(\mathbb{R}^{n})$ for all $t>0$. Based on this we can define the action of the heat semigroup and of the fractional powers (of order $s > 0$) of $-\nabla \cdot A \nabla$ as 
\[
\begin{aligned}
e^{t \nabla \cdot A \nabla} f &:= \int\limits_{0}^{\infty} e^{-t\lambda} \, \D E_{\lambda}f \quad \text{for all $f\in L^{2}(\mathbb{R}^{n})$ and $t > 0$,} \\
(-\nabla \cdot A \nabla)^{s} f &:= \int\limits_{0}^{\infty} \lambda^{s} \, \D E_{\lambda}f \quad \text{for all $f \in {\rm dom}\,\left( (-\nabla\cdot A\nabla)^{s} \right)$.}
\end{aligned}
\]
Using \cite[Theorems~7.6, 7.7 and 7.13]{Gri09HeatKernelManifold}, we know that the bounded operator $e^{t\nabla \cdot A \nabla}$ admits a \emph{unique} symmetric (heat) kernel $k_{t}^{A}(x,y)$: For each $0 <t < \infty$ and $f \in L^{2}(\mathbb{R}^{n})$, we have 
\begin{equation}
\big( e^{t\nabla \cdot A \nabla}f \big) (x) = \int\limits_{\mathbb{R}^{n}} k_{t}^{A}(x,y)f(y)\,\D y \quad \mbox{for all}\quad x\in \mathbb{R}^{n}, \label{eq:heat-kernel-anisotropic}
\end{equation}
see also \cite[Theorem~2]{QX21HeatKernel}. In addition, there exist positive constants $b_{1},b_{2},c_{1},c_{2}$ such that 
\begin{equation}
c_{1} t^{-\frac{n}{2}} \exp \bigg( -b_{1} \frac{|x-y|^{2}}{t} \bigg) \le k_{t}^{A}(x,y) = k_{t}^{A}(y,x) \le c_{2} t^{-\frac{n}{2}} \exp \bigg( -b_{2} \frac{|x-y|^{2}}{t} \bigg), \label{eq:Gaussian-bound}
\end{equation}
for all $t>0$ and $x,y \in \mathbb{R}^{n}$, see \cite[Chapter~3]{Dav89HeatKernelSpectral}. We refer to \cite{GT12HeatKernel} for the proof of two-sided estimates for heat kernels $k_{t}^{M}$ on abstract metric measure spaces $M$, extending those already known in Riemannian manifolds and in various types of fractals. Moreover, the heat kernel is also \emph{conservative} (or \emph{stochastically complete})\footnote{These terminologies are found in \cite{GHL14HeatKernel,GT12HeatKernel}. We also refer to \cite[Theorem~1.4.4]{Dav89HeatKernelSpectral} for an abstract result.}, that is,
\[
\int\limits_{\mathbb{R}^{n}} k_{t}^{A}(x,y) \, \D y = 1 \quad \text{for all $t>0$ and $x\in \mathbb{R}^{n}$.}
\]
In view of \eqref{eq:X-ray-heat-kernel}, now it is natural to consider the following definition. 

\begin{definition}[Weighted x-ray transform]\label{def:anisotropic-X-ray}
Let $n \ge 2$ be an integer, and let $0 < s < \frac{n}{2}$. For each $f \in \mathcal{S}(\mathbb{R}^{n})$, we define 
\begin{equation*}
\big( \mathcal{X}_{s}^{A}f \big) (x,\xi) := \int\limits_{0}^{\infty} w(x,x+\tau \xi) f(x + \tau \xi) \, \D \tau, 
\end{equation*}  
where the weight is given by 
\begin{equation*}
w(x,y) := \frac{1}{\Gamma(s) c(n,-s)} |y-x|^{n-1} \int\limits_{0}^{\infty} k_{t}^{A}(x,y) \, \frac{\D t}{t^{1-s}} \quad \mbox{for all $x \neq y \in \mathbb{R}^{n}$,}
\end{equation*}
and the average over the sphere $\mathbb{S}^{n-1}$ is defined by 
\begin{equation}
\big( \mathcal{A}_{s}^{A} f \big) (x) := c(n,-s) \int\limits_{\mathbb{S}^{n-1}} \big( \mathcal{X}_{s}^{A} f \big) (x,\xi) \, \D S_{\xi} \equiv \frac{1}{\Gamma(s)} \int_{\mathbb{R}^{n}} \bigg( \int\limits_{0}^{\infty} k_{t}^{A}(x,y) \, \frac{\D t}{t^{1-s}} \bigg) f(y) \,\D y. \label{eq:average-operator-anisotropic}
\end{equation}
\end{definition}

From \eqref{eq:Riesz-domain-range} and \eqref{eq:Gaussian-bound}, we can easily conclude the following lemma. 
\begin{lemma} \label{lem:boundedness-average-operator-anisotropic}
Let $n \ge 2$ be an integer and let $1 \le p < \frac{n}{2s}$ provided $0 < s < \frac{n}{2}$. Let $A(x) = (a_{ij}(x))_{i,j=1}^{n} \in (C^{\infty}(\mathbb{R}^{n}))^{n \times n}$ such that $a_{ij}=a_{ji}$ for all $i,j=1,\cdots,n$ and satisfies the ellipticity condition \eqref{eq:ellipticity}. Then 
\[
\mathcal{A}_{s}^{A} : L^{p}(\mathbb{R}^{n}) \rightarrow L^{q}(\mathbb{R}^{n}) \text{ is bounded,} 
\]
where $\frac{1}{q} = \frac{1}{p} - \frac{2s}{n}$. In particular when $0 < s < \frac{n}{4}$, the operators 
\[
\begin{aligned}
&\mathcal{A}_{s}^{A} : L^{2}(\mathbb{R}^{n}) \rightarrow L^{\frac{2n}{n-4s}}(\mathbb{R}^{n}) \\
&\mathcal{A}_{s}^{A} : L^{\frac{2n}{n+4s}}(\mathbb{R}^{n}) \rightarrow L^{2}(\mathbb{R}^{n})
\end{aligned}
\]
are bounded. 
\end{lemma}

Using the lemma above, we know that  
\[
(\mathcal{A}_{s}^{A}|f|)(x) = \frac{1}{\Gamma(s)} \int\limits_{\mathbb{R}^{n}} \int\limits_{0}^{\infty} \left| k_{t}^{A}(x,y) \frac{1}{t^{1-s}} f(y) \right| \, \D t \, \D y < \infty \quad \text{for a.e. $x \in \mathbb{R}^{n}$,}
\]
whenever $f \in L^{p}(\mathbb{R}^{n})$ for some $p \in [1,\frac{n}{2s})$. Therefore we can apply Fubini's theorem on \eqref{eq:average-operator-anisotropic} yields 
\begin{equation}
\big( \mathcal{A}_{s}^{A} f \big) (x) = \frac{1}{\Gamma(s)} \int\limits_{0}^{\infty} \bigg( \int\limits_{\mathbb{R}^{n}} k_{t}^{A}(x,y)f(y) \D y \bigg) \, \frac{\D t}{t^{1-s}} \equiv \frac{1}{\Gamma(s)} \int\limits_{0}^{\infty} (e^{t\nabla\cdot A\nabla}f)(x) \, \frac{\D t}{t^{1-s}} \label{eq:average-operator-anisotropic-semigroup}
\end{equation}
for all $f \in L^{p}(\mathbb{R}^{n})$ for some $p \in [1,\frac{n}{2s})$. In view of \eqref{eq:semigroup-definition-negative-Laplacian}, then we can define $(-\nabla\cdot A\nabla)^{-s}$ with a suitable domain rather than the artificial domain ${\rm dom}\,((-\nabla\cdot A\nabla)^{-s})$ given in \eqref{eq:domain-definition}. 

\begin{definition}\label{def:domain_of_anisotroipic_fractional_operator}
Let $n \ge 2$ be an integer and let $1 \le p < \frac{n}{2s}$ provided $0 < s < \frac{n}{2}$. Let $A(x) = (a_{ij}(x))_{i,j=1}^{n} \in (C^{\infty}(\mathbb{R}^{n}))^{n \times n}$ such that $a_{ij}=a_{ji}$ for all $i,j=1,\cdots,n$ and satisfies the ellipticity condition \eqref{eq:ellipticity}. Then the bounded linear operator $(-\nabla\cdot A\nabla)^{-s}: L^{p}(\mathbb{R}^{n}) \rightarrow L^{q}(\mathbb{R}^{n})$ with $\frac{1}{q} = \frac{1}{p} - \frac{2s}{n}$ is defined by 
\[
(-\nabla\cdot A\nabla)^{-s}f := \mathcal{A}_{s}^{A} f 
\]
for all $f \in L^{p}(\mathbb{R}^{n})$.
\end{definition}

\begin{remark}
The operator $(-\nabla\cdot A\nabla)^{-s}$ for $s \in \mathbb{C}$ with $\Re(s)>0$ also can be defined using Balakrishnan operator with domain ${\rm range}\,(A^{m})$ where $m$ is the minimum integer such that $m > \Re(s)$, see \cite[Definition~7.2.1]{MCSA01fractionalEllipticOperator}. 
\end{remark}

\bibliographystyle{custom}
\bibliography{ref}

\end{sloppypar}

\end{document}